\documentclass{article}[margin=0.5in]
\usepackage{graphicx} \usepackage{graphicx}
\usepackage{latexsym}
\usepackage{amssymb}
\usepackage{lineno}
\usepackage{enumerate}
\usepackage{amsmath}
\usepackage{amsthm}
\usepackage{color,soul}
\usepackage{authblk}
\usepackage{mathtools}
\usepackage{tikz}
\usetikzlibrary{positioning}
\usetikzlibrary{calc}
\usetikzlibrary{shapes,snakes}

\newcommand{\myradius}{3}
\newcommand{\mycirclesize}{16pt}
\newcommand{\rightvertexdist}{10} 
\newcommand{\rightvertexsize}{35pt} 

\newcommand{\drawlabels}[1][4]{
    \foreach \j in {1,2,...,#1}
    {
        \node[circle,draw,minimum size=\mycirclesize,fill=white] (\j) at ({(\j*360/#1)-19}:\myradius) {};
        \node at (\j) {\j};
    }
}

\definecolor{psychedelicpurple}{rgb}{0.87, 0.0, 1.0}
\definecolor{uclagold}{rgb}{1.0, 0.7, 0.0}

\makeatletter 
\tikzset{curlybrace/.style={rounded corners=2pt,line cap=round}}%
\pgfkeys{
/curlybrace/.cd,%
tip angle/.code     =  \def\cb@angle{#1},
/curlybrace/.unknown/.code ={\let\searchname=\pgfkeyscurrentname
                              \pgfkeysalso{\searchname/.try=#1,
                              /tikz/\searchname/.retry=#1}}}  
\def\curlybrace{\pgfutil@ifnextchar[{\curly@brace}{\curly@brace[]}}%

\def\curly@brace[#1]#2#3#4{%
\pgfkeys{/curlybrace/.cd,
tip angle = 0.75}%
\pgfqkeys{/curlybrace}{#1}%
\ifnum 1>#4 \def\cbrd{0.05} \else \def\cbrd{0.075} \fi
\draw[/curlybrace/.cd,curlybrace,#1]  (#2:#4-\cbrd) -- (#2:#4) arc (#2:{(#2+#3)/2-\cb@angle}:#4) --({(#2+#3)/2}:#4+\cbrd) coordinate (curlybracetipn);
\draw[/curlybrace/.cd,curlybrace,#1] ({(#2+#3)/2}:#4+\cbrd) -- ({(#2+#3)/2+\cb@angle}:#4) arc ({(#2+#3)/2+\cb@angle} :#3:#4) --(#3:#4-\cbrd);
}

\def\curlybraceflipped[#1]#2#3#4{%
\pgfkeys{/curlybrace/.cd,
tip angle = 0.75}%
\pgfqkeys{/curlybrace}{#1}%
\ifnum 1>#4 \def\cbrd{0.05} \else \def\cbrd{0.075} \fi
\draw[/curlybrace/.cd,curlybrace,#1]  (#2:-#4-\cbrd) -- (#2:-#4) arc (#2:{(#2+#3)/2-\cb@angle}:-#4) --({(#2+#3)/2}:-#4+\cbrd) coordinate (curlybracetipn);
\draw[/curlybrace/.cd,curlybrace,#1] ({(#2+#3)/2}:-#4+\cbrd) -- ({(#2+#3)/2+\cb@angle}:-#4) arc ({(#2+#3)/2+\cb@angle} :#3:-#4) --(#3:-#4-\cbrd);
}
\makeatother

\newcommand{\Mod}[1]{\ (\mathrm{mod}\ #1)}

\newtheorem{theorem}{Theorem}
\newtheorem{question}{Question}
\newtheorem{lemma}{Lemma}
\newtheorem{observation}{Observation}
\newtheorem{corollary}{Corollary}
\newtheorem{example}{Example}

\begin{document}

\title{Embedding equitable $(s,p)$-edge-colorings of $K_n$}
\author{}
\date{}

\maketitle

\vspace{-50pt}

\begin{center}

\begin{tabular}{ccc} \large Stacie Baumann$^1$ & \large Mika Olufemi$^1$ & \large Stafford Yerger$^{1,2}$ \\ \texttt{\small baumannsm@cofc.edu} & \texttt{\small olufemimo@g.cofc.edu} & \texttt{\small staffordyerger@gmail.com} \end{tabular}\\ \vspace{8pt} \small $^1$ Department of Mathematics, College of Charleston, SC, USA\\ $^2$ Department of Physics and Astronomy, Louisiana State University, LA, USA \end{center}

\begin{abstract}
 An $(s,p)$-edge-coloring of a graph $G$ is an edge coloring using $s$ colors such that exactly $p$ colors appear at each vertex. To generalize the notion of proper edge-coloring, these colorings are defined to be equitable: the numbers of edges of each color incident to a vertex are fairly distributed. We find the necessary and sufficient conditions for embedding an equitable $(s_1,p_1)$-edge-coloring of $K_{n_1}$ into an equitable $(s_2,p_2)$-edge-coloring of $K_{n_2}$. We focus on the values of $n_1$, $p_1$, $n_2$, and $p_2$ where $s_1$ is necessarily larger than $p_1$ and $s_2$ is necessarily larger than $p_2$.
\end{abstract}




\section{Introduction}

We motivate this paper with a scheduling problem similar to the one discussed by Baumann and Rodger in \cite{mine}.
Suppose a schedule was set for $n_1=25$ people to meet $p_1=8$ hours and take a break for $s_1-p_1=1$ hour during a $s_1=9$ hour day of a convention.
Each person is scheduled to meet one-on-one with $(n_1-1)/p_1=3$ people per hour.
However, now $6$ more people decide to come to the convention. Now there are $n_2=31$ people.
Can we form a schedule so each person meets $p_2=10$ out of the $s_2=11$ hours of the day and they still meet with $(n_2-1)/p_2=3$ people per hour without changing the previous schedule?

The schedules of this form can be modeled with the complete graph $K_{n_i}$ where each vertex represents a person.
The edges of $K_{n_i}$ are colored with $s_i$ colors where an edge $\{u,v\}$ is colored $k$ if person $u$ and $v$ meet during the $k$th hour of the conference.
Because each person is scheduled to meet during $p_i$ hours and the meeting should be spread as equally as possible during those hours, each vertex is incident to $n_i-1/p_i$ edges of each of the $p_i$ colors.
As we will define shortly, this problem asks if an equitable $(s_1,p_1)$-edge-coloring of $K_{n_1}$ can be embedded in an equitable $(s_2,p_2)$-edge-coloring of $K_{n_2}$. 
Note that completing this schedule would require scheduling the one-on-one meetings within each hour, but that is not the focus of this paper.

We say that $W: E(G) \mapsto C = \{1, 2, \dots, s\}$ is an \emph{$(s,p)$-edge-coloring} of a graph $G$ if:
\begin{enumerate}
    \item the edges in $E(G)$ are colored with exactly $s$ colors, and
    \item for each vertex $u \in V(G)$ the edges incident to $u$ are colored using exactly $p$ colors.
\end{enumerate}
    
\noindent{The ($s$,$p$)-edge-coloring, $W$, is said to be \emph{equitable} if}

\begin{enumerate}
\setcounter{enumi}{2}
    \item for each vertex $u \in V(G)$ and for each $\{i,j\} \subseteq C(W,u)$, $\vert b(W, u, i) - b(W, u, j) \vert \leq 1$,
\end{enumerate}
where $C(W,u) = \{i \mid W \text{ colors some edge incident with } u \text{ with color } i\}$, and $b(W, u, i)$ is the number of edges incident to $u$ that are colored $i$ by $W$.
Note that this definition of equitable generalizes the definition used in the usual edge-coloring situation where $s = p$.

These edge-colorings were first defined in terms of equitable $(s,p)$-block-colorings.
There are several papers studying the existence and structure equitable $(s,p)$-block-colorings and equitable $(s,p)$-edge-colorings. See \cite{mine, MR2007895, MR2684082, MR2438863, MR2629918, MR3774803, MR3475564, MR3764373}.
Also, \cite{baumannthesis} summarizes the results.

We say an equitable $(s_1,p_1)$ edge-coloring of $G_1$, say $W_1$, can be \emph{embedded} in an equitable $(s_2,p_2)$-edge-coloring of $G_2$, say $W_2$, if there exists an equitable $(s_2,p_2)$-edge-coloring of $G_2$ with a proper subgraph of $G_1$ with coloring $W_1$.
Note that this is different from embedding a graph in the plane and using the word ``embed" here follows the convention used in design theory.

A more specific version of the following question was asked in \cite{baumannthesis}.

\begin{question}
\label{mainquest}
When can an equitable $(s_1, p_1)$-edge-coloring of $K_{n_1}$ be embedded in an equitable $(s_2, p_2)$-edge-coloring of $K_{n_2}$?
\end{question}

In this paper, we find the necessary and sufficient conditions to answer this question for an interesting case where $s_1$ and $s_2$ are necessarily larger than $p_1$ and $p_2$, respectively, and the number of edges of each color at each vertex is preserved.
The formal definition of this case is provided in Section \ref{subsectdefandprevresults}.
Note that the scheduling problem discussed at the beginning of this section corresponds to a specific example of one of these embeddings.
Other cases of Question \ref{mainquest} are discussed in Section \ref{futuresection}.

Embedding or extending a mathematical object to another mathematical object is widely studied.
There are many results about embedding designs.
Some of the first of these results were Ryser and Cruse's results about latin squares.
Ryser \cite{Ryser} and Cruse \cite{cruse1974embedding} each proved embedding results about latin squares that are equivalent to finding necessary and sufficient conditions for embedding a proper edge-coloring of the complete bipartite graph $K_{r,s}$ into a proper edge-coloring of $K_{n,n}$ and embedding a proper complete graph with a loop on every vertex $\mathbb{K}_r$ into a proper edge-coloring of $\mathbb{K}_n$ (where a color on the loop contributes $1$ to the degree of the vertex to not contradict the coloring being proper).
As discussed in Lemma \ref{thm:b(n)=1}, Cruse's Theorem \cite{cruse1974embedding} for embedding symmetric latin squares actually proves a special case of our main result.
These results started a long study of embedding designs;
For more embedding results, see the following survey and books \cite{colbourn2010crc, designbook, rodgersurvey}.

Extensions of both vertex and edge-colorings of graphs has also been studied.
We discuss a few recent results about extending edge-colorings.
Hilton and Andersen \cite{andersen1989extending} and Hendersen and Hilton \cite{henderson2004completing} studied extending proper edge-colorings of $K_{n_1}$ and independent edges to a proper edge-coloring of $K_{n_2}$.
Marcotte and Seymour found a necessary condition for extending a proper $k$-edge-coloring of a subgraph of $G$ to a proper $k$-edge-coloring of $G$ and determined some classes of graphs that this is sufficient.
Edwards et. al. considered when a matching in $G$ can be extended to proper-edge-coloring of $G$ \cite{edwards2018extension}.
Harrelson et. al. proved results about properly coloring a graph with precolored edges in planar graphs \cite{harrelson2019list}.
There have also been several results about extending a partial $r$-factorization of of $K_{n_1}$ into $r$-factorizations of $K_{n_2}$.
Most recently Bahmanian, Johnsen, and Napirata proved several results and provide a history of this problem in \cite{bahmanian2022embedding, bahmanian2025embedding}.

\subsection{Definitions and Previous Results}
\label{subsectdefandprevresults}

Let $\chi'_{p}(K_n)$ be the smallest integer $s$ such that there exists an equitable $(s,p)$-edge-coloring of $K_n$.
The following theorem highlights an interesting case where $\chi'_{p}(K_n)>p$.
For a summary of how to combine the results cited to form the theorem, see \cite{baumannthesis}.

\begin{theorem} \cite{MR1271864, MR3774803, MR3475564} \label{mainlambda1} For $1 \leq p \leq n-1$,
$$
\chi'_p(K_n) =
\begin{cases} 
      p+1 & \text{ if } $p$ \text{ even and } n \equiv p+1 \Mod{2p}\\
      p & \text{ otherwise.}
   \end{cases}
$$
\end{theorem}

So when $p$ even and $n \equiv p+1 \Mod {2p}$, $s$ is necessarily larger than $p$.
This case is interesting because when $s>p$ the typical edge-coloring approach of interchanging colors along a $2$-edge-colored path is not viable.

Let $W$ be a $(s,p)$-edge-coloring of $K_n$. Define $b(n, W, u, i)$ as the number of edges incident to $u$ that are colored $i$ by $W$.
Matson and Rodger proved the following.

\begin{observation} \cite{MR3764373}
\label{bobservation}
Let $W$ be a $(2t+1,2t)$-edge-coloring of $K_n$ with the colors in $C$ for $n=4tx+2t+1$ for some integer $x$.
Then $b(n, W, u, i)= \frac{n-1}{2t}=2x+1$ for all $u \in V(K_n)$ and $i \in C$.
\end{observation}

By the previous observation, $b(n, W, u, i)$ is independent of $u$ and $i$.

As stated in the previous section, we find the necessary and sufficient conditions to answer Question \ref{mainquest} for an interesting case where $s_1$ and $s_2$ and necessarily larger than $p_1$ and $p_2$, respectively, and the number of edges of each color at each vertex is preserved.
Therefore, in light of Theorem \ref{mainlambda1}, for the remainder of the paper we will assume $n_1 \equiv 2t_1 + 1 \mod 4t_1$ and $n_2 \equiv 2t_2 + 1 \mod 4t_2$ for some integers $1 \leq t_1 < t_2$.
Also, by Observation \ref{bobservation} and because we would like to preserve the number of edges of each color at each vertex, for the remainder of the paper we will assume $n_1=p_1b+1$ and $n_1=p_2b+1$ for some odd integer $b$.
Therefore, we answer the following question.

\begin{question}
\label{mainquestbpreserved}
Let $1 \leq t_1 < t_2$ be integers and $b$ be an odd integer.
When can an equitable $(2t_1+1, 2t_1)$-edge-coloring of $K_{n_1}$ be embedded in an equitable $(2t_2+1, 2t_2)$-edge-coloring of $K_{n_2}$ where $n_1=2t_1b+1$ and $n_2=2t_2b+1$?
\end{question}

In summary, assume the following throughout the remainder of the paper.
\begin{itemize}
\item $1 \leq t_{1} < t_{2}$ are integers
\item $b \geq 1$ is an odd integer
\item $p_i = 2t_i$, $s_i=2t_i+1$, and $n_i = 2t_ib+1$ for $1 \leq i \leq 2$
\item $V_{O} \coloneqq \{v_1, v_2, \dots v_{n_1}\}$ and $V_{N} \coloneqq \{v_{n_1+1}, v_{n_1+2}, \dots v_{n_2}$\}
(Note that $V_O$ are the ``Old vertices" and $V_N$ are the ``New vertices")
\item $C_O \coloneqq \{1, 2, \dots, s_1\}$ and $C_N := \{s_1+1,s_1+2,...,s_2\}$ (Note that the colorings in $C_O$ are the ``Old colors" and in $C_N$ are the ``New colors")
\item $m \coloneqq \vert C_N \vert =s_2-s_1$ (So, $m$ is the ``number of new colors")
\item $V(K_{n_2}) = \{v_1, v_2, \dots, v_{n_2}\}$ and the edges of the graph induced by the vertices $\{v_1, v_2, \dots, v_{n_1}\}$ have an arbitrary equitable $(s_1, p_1)$-edge-coloring of $K_{n_1}$ with the colors in $C_O$
\end{itemize}




We now discuss an observation and lemma that will be used frequently throughout the paper.

We say a vertex is \emph{missing} a color if there is no edge incident to that vertex colored with that color.
We use the following observation frequently throughout the paper without referencing it.

\begin{observation} \label{missingone}
In an equitable $(p+1,p)$-edge-coloring of a graph $G$, each $v \in V(G)$ is missing exactly one color.
\end{observation}

The following term was introduced in \cite{MR3774803}.
The \emph{color vector} of an equitable $(s,p)$-edge-coloring $W$ of a graph $G$ is the vector $V_W=(c_1(W),c_2(W), \dots, c_s(W))$ in which, for $1 \leq i \leq s$, $c_i(W)$ is the number of vertices in $G$ that are incident with an edge of color $i$; by renaming colors if necessary, it is always assumed that $c_1(W) \leq c_2(W) \leq \dots \leq c_s(W)$.

\begin{lemma}[\cite{MR3774803}]\label{colorvectorlemma} Let $n=2tb+1$ for some odd positive integer $b$ and some positive integer $t$, $p=2t$, and $s=2t+1$. In any equitable $(s,p)$-edge-coloring $W$ of $K_n$, for $1 \leq i \leq s$,
\begin{enumerate}[(i)]
\item $c_i(W)$ must be even,
\item $c_i(W) \geq b +1 = \frac{n-1}{p} +1$, and
\item $c_i(W) \leq n-1$.
\end{enumerate}
\end{lemma}

\section{Necessity}

We first provide the necessary condition.

\begin{lemma}
\label{ineqlem}
    Let $W_1$ be an equitable $(s_1, p_1)$-edge-coloring of $K_{n_1}$. If we can embed $W_1$ into an equitable $(s_2, p_2)$-edge-coloring of $K_{n_2}$, $W_2$, then 
    $$
        n_1(m-1) \leq (n_2-n_1)m-1.
    $$
\end{lemma}

\begin{proof}
    Define $E_{C_N}(V_O,V_N)$ to be the set of edges with one endpoint in $V_O$ and one endpoint in $V_N$ that are colored with a color in $C_N$.
    Recall, by Observation \ref{missingone}, that each vertex is missing exactly $1$ color.
    Let $N$ be the set of vertices in $V_O \cup V_N$ that are missing a color in $C_N$.

    Because each of the vertices in $N$ must be incident with $b$ edges of $m-1$ colors in $C_N$ and each of the vertices in $(V_O \cup V_N) \setminus N$ must be incident with $b$ edges of $m$ colors in $C_N$,
    \begin{equation}
        \label{oldvertineq}
    |E_{C_N}(V_O,V_N)|=|V_O \cap N|(m-1)b + |V_O \setminus N|mb
    \end{equation}
    and
    \begin{equation}
        \label{newvertineq}
        |E_{C_N}(V_O,V_N)| \leq |V_N \cap N|(m-1)b + |V_N \setminus N|mb.
    \end{equation}
    Note that the first equation is an equality because there are no colors from $C_N$ on the edge $\{v,u\}$ for all $v,u \in V_O$, but the second equation is an inequality because there may be colors from $C_N$ on the edge $\{v,u\}$ for some $v,u \in V_N$.
    Immediately, we can conclude from the previous two equations that $n_1(m-1)b \leq |E_{C_N}(V_O,V_N)| \leq (n_2 - n_1)mb$.
    However, this inequality can be improved.

    By Lemma \ref{colorvectorlemma} (i), $c_i(W_2)$ must be even for all $1 \leq i \leq s_2$. Also, $n_2$ is odd. So, the number of vertices missing each color $i \in C_N$ is odd.
    Additionally, $|C_N|=m$ is even.
    Therefore, $\vert N \vert$ is even.
    Consider $2$ cases.
    
    \noindent {\bf Case 1:} 
    Assume $|V_O \cap N|$ is odd.
    Therefore, $|V_N \cap N| \geq 1$.
    So, by (\ref{oldvertineq}) and (\ref{newvertineq}), $$n_1(m-1)b \leq |E_{C_N}(V_O,V_N)| \leq (1)(m-1)b+(n_2-n_1-1)mb.$$

    \noindent {\bf Case 2:} 
    Assume $\vert V_O \cap N \vert$ is even. 
    Because $\vert V_O \vert$ is odd, there must be at least one vertex in $V_O$ missing a color in $C_O$.
    Therefore, $|V_O \setminus N| \geq 1$.
     So, by (\ref{oldvertineq}) and (\ref{newvertineq}), $$(n_1-1)(m-1)b + (1)mb \leq |E_{C_N}(V_O,V_N)| \leq (n_2-n_1)mb.$$
     
     Thus, in either case, $n_1(m-1) \leq (n_2-n_1)(m)-1.$
 \end{proof}

\begin{corollary}\label{equitableembedinequality}
Let $W_1$ be an equitable $(2t_1+1, 2t_1)$-edge-coloring of $K_{n_1}$. If we can embed $W_1$ into an equitable $(2t_2+1, 2t_2)$-edge-coloring of $K_{n_2}$ called $W_2$, then $t_2 \geq 2t_1$.
\end{corollary}

\begin{proof}
By Lemma \ref{ineqlem}, $n_1(m-1)\leq (n_2-n_1)(m)-1$, so
$$
\begin{aligned}
0  & \leq (n_2-n_1)(m)-n_1m+n_1-1\\
  & = bm^2 -(2t_1b+1)m+2t_1b.\\
\end{aligned}
$$
Suppose for a contradiction that $t_2 \leq 2t_1-1$.
Because $t_2 > t_1$, $t_1 \geq 2$.
Recall, $m=s_2-s_1=2t_2-2t_1$.
Therefore, $m \leq 2t_1-2$.
Thus,
$$
\begin{aligned}
0  & \leq b(2t_1-2)^2 -(2t_1b+1)(2t_1-2)+2t_1b\\
& =4t_1^2b-8t_1b+4b-4t_1^2b+4t_1b-2t_1+2+2t_1b\\
& = -2t_1b+4b-2t_1+2\\
& < (-2t_1+4)(b+1).\\
\end{aligned}
$$
Because $t_1 \geq 2$ and $b \geq 1$, this is a contradiction.
Thus, $t_2 \geq 2t_1$.
\end{proof}

\begin{corollary}\label{equitableembedinequalitywithn}
Let $W_1$ be an equitable $(2t_1+1, 2t_1)$-edge-coloring of $K_{n_1}$. If we can embed $W_1$ into an equitable $(2t_2+1, 2t_2)$-edge-coloring of $K_{n_2}$ called $W_2$, then $n_2 \geq 2n_1-1$.
\end{corollary}


In the following section, we show this necessary condition is also sufficient.

\section{Sufficiency}
\label{sufficiencysection}

Cruse's Theorem \cite{cruse1974embedding} for embedding symmetric latin squares proves sufficiency when $b=1$.
Therefore, we have the following lemma.

\begin{lemma}
\label{thm:b(n)=1}
If $b=1$, an equitable $(2t_1+1, 2t_1)$-edge-coloring of $K_{n_1}$
can be embedded in an equitable $(2t_2+1, 2t_2)$-edge-coloring of $K_{n_2}$ if and only if $t_1 \leq 2t_2$.
\end{lemma}

See \cite{Bachessay} for information on how Cruse's Theorem can be used to prove the previous lemma and an alternative proof of the $b=1$ case using similar methods to those below.
For the remainder of this section, assume $b\geq 3$.

In Section \ref{introamalgsection}, we begin by giving background information needed. In Section \ref{subsectt2=2t1} and Section \ref{subsectt2>2t1}, we prove the sufficiency in two cases, $t_2=2t_1$ and $2t_1 < t_2 < 4t_1$, respectively.
In Section \ref{Intuition}, we discuss why we consider two separate cases.
In Section \ref{mainresultsection}, the two cases are combined to prove our main result.

\subsection{The Proof Technique of Amalgamations}
\label{introamalgsection}

The \emph{amalgamation} of a graph $G$ with amalgamation function $\phi: V(G) \rightarrow V(G')$ is the graph $G'$, which can have loops and multi-edges, with the vertex set $V(G')$ and edge multiset $E(G') = \{\{\phi(x_i),\phi(x_j)\} | \{x_i,x_j\} \in E(G)\}$, where $\{\phi(x_i),\phi(x_i)\}$   denotes a loop on vertex $\phi(x_i)$. For each $x_i \in V(G')$, we call $\eta$ the number function associated with $\phi$, such that $\eta(x_i)= \vert\phi^{-1}(x_i)\vert$.
So, $\eta(x_i)$ is the number of vertices in $V(G)$ that were amalgamated to form $x_i$. 
Let $\ell(x_i)$ denote the number of loops on the vertex $x_i$ where each loop contributes two to the degree on $x_i$ and let $x \approx y$ be $\lfloor y \rfloor \leq x \leq \lceil y \rceil$.

Bahmanian and Rodger proved the following theorem.
\begin{theorem} [\cite{MR2946077}] \label{amalgamationtheorem} Let $G'$ be a $k$-edge-colored graph and let $\eta$ be a function from $V(G')$ into $\mathbb{N}$ such that for each $w \in V(G')$, $\eta(w)=1$ implies $\ell_{G'}(w)=0$.
Then there exists a loopless $\eta$-detachment $G$ of $G'$ with amalgamation function $\phi: V(G) \rightarrow V(G')$, $\eta$ being the number function associated with $\phi$, such that $G$ satisfies the following conditions:
\begin{enumerate}[(i)]
    \item $deg_{G(j)}(u) \approx deg_{G'(j)}(w) / \eta(w)$ for each $w \in V(G')$, each $u \in \phi^{-1}(w)$, and for $1 \leq j \leq k$;
    \item $m_G(u, u') \approx \ell_{G'}(w)/ {\eta(w) \choose 2}$ for each $w \in V(G')$ with $\eta(w) \geq 2$ and every pair of distinct vertices $u,u' \in \phi^{-1}(w)$; and
    \item $m_G(u,v) \approx m_{G'}(w,z) / (\eta(w) \eta(z))$ for every pair of distinct vertices $w,z \in V(G')$, each $u \in \phi^{-1}(w)$, and each $v \in \phi^{-1}(z)$.
\end{enumerate}
\end{theorem}

In our proofs of Lemma \ref{obs:Old-color-method-theorem} (see Section \ref{subsectt2=2t1}) and Lemma \ref{obs:New-color-method-theorem} (see Section \ref{subsectt2>2t1}), we define a convenient amalgamation on $K_{n_2}$ 
and color the edges such that after detachment by the previous theorem, we have an equitable $(2t_2+1,2t_2)$-edge-coloring of $K_{n_2}$ with an arbitrary equitable $(2t_1+1,2t_1)$-edge-coloring of $K_{n_1}$ embedded.

\subsection{Sufficiency when $t_2=2t_1$}
\label{subsectt2=2t1}

We first consider the case where $t_2=2t_1$.
We have included Example \ref{oldcolormethodex} following the proof to help the reader.

Let $\lambda G$ denote the graph obtained by replacing every edge in $G$ with $\lambda$ copies of the edge.

\begin{lemma}\label{obs:Old-color-method-theorem}
Let $t_2 = 2t_1$.
An equitable $(2t_1+1, 2t_1)$-edge-coloring of $K_{n_1}$ can be embedded in an equitable $(2t_2+1, 2t_2)$-edge-coloring of $K_{n_2}$.
\end{lemma}

\begin{proof}
Let $X = \{x_1, x_2, \dots x_{n_1}\}$.
We will form edge-colored graphs $G_1, G_2, G_3$ on the vertex set $X \cup \{y\}$ as follows.
Once this is done, defining $G'= \cup_{i=1}^3 G_i$, Theorem \ref{amalgamationtheorem} will be used to detach $G'$ to form $G= K_{n_2}$ with our desired equitable $(2t_2+1,2t_2)$-edge-coloring and an arbitrary $(2t_1+1, 2t_1)$-edge-coloring of $K_{n_1}$ as a subgraph.
To do so, $\phi: V(G) \rightarrow V(G')$ will be the amalgamation function with amalgamation numbers $\eta(x_i)= 1$ for $1 \leq i \leq n_1$ and $\eta(y)= n_2 - n_1$.

\begin{enumerate}[i)]
\item Let $G_1=K_{n_1}$ on the vertex set $X$ with any equitable $(2t_1+1, 2t_1)$-edge-coloring, call it $W_1$, using the colors in $C_O$ (such a coloring exists by Theorem \ref{mainlambda1}).
\item Let $S_w$ be the set of vertices in $X$ missing color $w$ in $W_1$ for $1 \leq w \leq 2t_1+1$.
Note that $\vert S_w \vert$ is odd and $\vert S_w \vert \geq 1$ by Lemma \ref{colorvectorlemma} (i) and (iii).
We will form sets $S_w'$ for $1 \leq w \leq 2t_2+1$ by modifying each set $S_w$. Without loss of generality, assume $|S_1|\leq|S_w|$ for $2\leq w\leq 2t_1+1$.
Define $S_1' \coloneqq S_1$. Remove $2t_1$ vertices from some sets $S_w$ for $2 \leq w \leq 2t_1+1$ keeping each $S_w$ odd (repeatedly remove pairs).
Call these new sets (after the vertices have been removed) $S'_w$.
This is possible because

\[
\begin{aligned}
n_1- (2t_1) &= 2t_1b+1-2t_1\\
 &= 2t_1(b-1) +1\\
 &\geq 2t_1+1 &\text{Because $b\geq 3$.}
\end{aligned}
\]
Therefore, $\vert S_w' \vert \geq 1$ for $1 \leq w \leq 2t_1+1$.
Greedily place the removed vertices in singleton sets $S'_w$ for $2t_1+1 \leq w \leq 2t_2+1 = 4t_1+1$.
Let $G_2 = (n_2-n_1)K_{n_1, 1}$ with parts $X$ and $\{y\}$ by attaching each vertex $x \in S_w \cap S'_c$ for $1 \leq w \leq 2t_1+1$ and $1 \leq c \leq 2t_2+1$ to $y$ with $b$ edges colored $k$ for  
$k \in (C_N \cup \{w\}) \setminus \{c\}$.
\item Form $G_3$ by placing $n_2-n_1 \choose 2$ loops on a single vertex, $y$.
For each $k \in C_O \setminus \{1\}$, color $\frac{(n_2-n_1)b-\vert S_k \setminus S'_k \vert b}{2}$ loops with color $k$.
This is well-defined because
\begin{equation*}
\begin{array}{rcl}
\ell_{G'}(y) & = & \sum_{k=2}^{2t_1+1} \frac{(n_2-n_1)b-\vert S_k\setminus S_{k}' \vert b}{2} \\[4pt]
& = & \sum_{k=2}^{2t_1+1} \frac{((2t_{2}b+1)-(2t_{1}b+1))b-\vert S_k\setminus S_{k}' \vert b}{2}\\[4pt]
& = & \sum_{k=2}^{2t_1+1} \frac{2t_1b^2}{2}-\sum_{k=2}^{2t_1+1} \frac{\vert S_k\setminus S_{k}' \vert b}{2}\\[4pt]
& = & \frac{
(2t_1b^2)(2t_1)}{2} - \frac{2t_1b}{2}\\[4pt]
& = & \frac{2t_1b(2t_1b-1)}{2}\\[4pt]
& = & \frac{(2t_2-2t_1)(b)((2t_2-2t_1)b-1)}{2} \\[4pt]
& = & \frac{((2t_2b+1)-(2t_1b+1))((2t_2b+1)-(2t_1b+1)-1)}{2}\\[4pt]
& = & \frac{(n_2-n_1)(n_2-n_1-1)}{2} \\[4pt]
& = & {n_2-n_1 \choose 2}. \\[4pt]
\end{array}
\end{equation*}
\end{enumerate}

As stated above, define $G'=\cup_{i=1}^3 G_i$.
We will use Theorem \ref{amalgamationtheorem} to form the detachment $G= K_{n_2}$ with our desired equitable $(2t_2+1,2t_2)$-edge-coloring and an arbitrary $(2t_1+1, 2t_1)$-edge-coloring of $K_{n_1}$ as a subgraph.
Recall that $\phi: V(G) \rightarrow V(G')$ is the amalgamation function with amalgamation numbers $\eta(x_i)= 1$ for $1 \leq i \leq n_1$ and $\eta(y)= n_2 - n_1$.

We first check that there are $b$ edges of each color incident to each vertex $v_i \in V_O$ except exactly $1$ color. Suppose $x_i \in S'_{c}$, then $v_i$ is missing color $c$.

\begin{itemize}
\item For $k \in C_O \setminus \{c\}$ and $1 \leq i \leq n_1$, by (i) or (ii), $deg_{G'(k)}(x_i) = b$.
By Theorem \ref{amalgamationtheorem} (i), $deg_{G(k)}(v_i) = deg_{G'(k)}(x_i)/\eta(x_i) = b/1 = b$ for $v_i \in \phi^{-1} (x_i)$.
\item For $k \in C_N \setminus \{c\}$ and $1 \leq i \leq n_1$, by (ii), $deg_{G'(k)}(x_i) = b$.
By Theorem \ref{amalgamationtheorem} (i), $deg_{G(k)}(v_i) = deg_{G'(k)}(x_i)/\eta(x_i) = b/1 = b$ for $v_i \in \phi^{-1} (x_i)$.
\end{itemize}

We will now check that each vertex detached from $y$ is incident to $b$ edges of each color $k\in (C_O\cup C_N) \setminus \{1\}$.
\begin{itemize}
\item For $2t_1 + 2 \leq k \leq 2t_2 + 1$, by (ii),
\[
\begin{aligned}
deg_{G'(k)}(y) &= (n_1 - 1)b \\
&= 2t_1 b^2 \\
&= (4t_1 b + 1 - (2t_1 b + 1))b \\
&= (2t_2 b + 1 - (2t_1 b + 1))b \\
&= (n_2 - n_1)b.
\end{aligned}
\]
So $deg_{G(k)}(v_i) = deg_{G'(k)}(y)/\eta(y) = (n_2-n_1)b/(n_2-n_1)=b$ for $v_i \in \phi^{-1} (y)$ by Theorem \ref{amalgamationtheorem} (i).
\item For $2 \leq k \leq 2t_1 + 1$, by (ii) and (iii),
\[
\begin{aligned}
deg_{G'(k)}(y) &= |S_k \setminus S_{k}'|b + 2 \left(\frac{(n_2 - n_1)b - (|S_k \setminus S_{k}'|b)}{2}\right) \\
&= (n_2-n_1)b.
\end{aligned}
\]
By Theorem \ref{amalgamationtheorem} (i), $deg_{G(k)}(v_i) = deg_{G'(k)}(y)/\eta(y) = (n_2-n_1)b/(n_2-n_1)=b$ for $v_i \in \phi^{-1} (y)$.
\end{itemize}

We now check that $G=K_{n_2}$, by determining there is exactly $1$ edge between each pair of vertices in $V(G)$.
\begin{itemize}
\item We first show there is exactly $1$ edge between each pair of vertices $v_i,v_j \in V_O$.
For $x_i, x_j \in X$, by (i), $m_{G'}(x_i, x_j) =1$.
By Theorem \ref{amalgamationtheorem} (iii), $m_G(v_i, v_j) = m_{G'}(x_i, x_j)/(\eta(x_i)\eta(x_j)) = 1/ ((1)(1)) = 1$ for $v_i \in \phi^{-1} (x_i)$ and $v_j \in \phi^{-1} (x_j)$.
\item We now show there is exactly $1$ edge between each pair of vertices $v_i,v_j \in V_N$.
By (iii), $\ell_{G'}(y) = {n_2-n_1 \choose 2}={\eta(y) \choose 2}$.
By Theorem \ref{amalgamationtheorem} $(ii)$, $m_G(v_i, v_j) = \ell_{G'}(y)/ {\eta(y) \choose 2}= {\eta(y) \choose 2}/ {\eta(y) \choose 2}=1$.
\item Finally, we show there is exactly $1$ edge between each pair of vertices $v_i \in V_O$ and $v_j \in V_N$.
By (ii), for each $x_i \in X$, $ m_{G'}(x_i, y)=b \vert (C_N \cup \{w\}) \setminus \{c\}\vert = b(2t_1)$ for some $1 \leq w \leq 2t_1+1$ and some $1 \leq c \leq 2t_2+1$.
So,
\[
\begin{aligned}
m_G(v_i, v_j) &= \frac{m_{G'}(x_i, y)}{(\eta(x_i)\eta(y))} \\[3pt]
&= \frac{2bt_1}{(1)(n_2-n_1)} \\[3pt]
&= \frac{2bt_1}{(2t_2b+1)-(2t_1b+1)}\\[3pt]
&= \frac{2bt_1}{(4t_1b+1)-(2t_1b+1)} \\[3pt]
&= \frac{2bt_1}{2bt_1}\\[3pt]
&= 1 
\end{aligned}
\]
for $v_i \in \phi^{-1}(x_i)$ and $v_j \in \phi^{-1}(y)$ by Theorem $\ref{amalgamationtheorem}$ (iii).
\end{itemize}
\end{proof}

The following is an example of an embedding that can be found using the method in Lemma \ref{obs:Old-color-method-theorem}.

\begin{example}
\label{oldcolormethodex}
An equitable $(3,2)$-edge-coloring of $K_7$ can be embedded in an equitable $(5,4)$-edge-coloring of $K_{13}$.
\end{example}


\begin{proof}

Let $X = \{x_1, x_2, \dots x_{7}\}$.
We will form edge-colored graphs $G_1, G_2, G_3$ on the vertex set $X \cup \{y\}$ as follows.
Once this is done, defining $G'= \cup_{i=1}^3 G_i$, Theorem \ref{amalgamationtheorem} will be used to detach $G'$ to form $G= K_{13}$ with our desired equitable $(5,4)$-edge-coloring and an arbitrary $(3,2)$-edge-coloring of $K_{7}$, $W_1$, as a subgraph (see Figure \ref{expic} (a-b)).
To do so, $\phi: V(G) \rightarrow V(G')$ will be the amalgamation function with amalgamation numbers $\eta(x_i)= 1$ for $1 \leq i \leq 7$ and $\eta(y)= 6$.
\begin{enumerate}[i)]
\item  Let $G_1$ on the vertex set $X$ be a $K_7$ with coloring $W_1$ (see Figure \ref{expic} (c)). 
\item Form $G_2=6K_{7,1}$ with parts $X$ and $\{y\}$ and color edges as described in Lemma \ref{obs:Old-color-method-theorem} (ii) (see Figure \ref{expic} (d)). 
\item Form $G_3$ by placing $15$ loops on $y$ and coloring as described in Lemma \ref{obs:Old-color-method-theorem} (iii) (see Figure \ref{expic} (e)).
\end{enumerate}
Let $G'= \cup_{i=1}^3 G_i$.
Use Theorem \ref{amalgamationtheorem} to detach vertices from $y$, creating $G=K_{13}$ with an equitable $(5,4)$-edge-coloring (See Figure \ref{expic} (b)).
\end{proof}

\begin{figure}
\centering
\begin{tabular}{ccc}
\begin{tikzpicture}
[scale=0.6, every node/.style={fill=black, circle, font=\scriptsize, very thin, scale=.4}]
\drawlabels[13]

\draw[line width=1pt, gray!30] (1) -- (8);
\draw[line width=1pt, gray!30] (1) -- (9);
\draw[line width=1pt, gray!30] (1) -- (10);
\draw[line width=1pt, gray!30] (1) -- (11);
\draw[line width=1pt, gray!30] (1) -- (12);
\draw[line width=1pt, gray!30] (1) -- (13);

\draw[line width=1pt, gray!30] (2) -- (8);
\draw[line width=1pt, gray!30] (2) -- (9);
\draw[line width=1pt, gray!30] (2) -- (10);
\draw[line width=1pt, gray!30] (2) -- (11);
\draw[line width=1pt, gray!30] (2) -- (12);
\draw[line width=1pt, gray!30] (2) -- (13);

\draw[line width=1pt, gray!30] (3) -- (8);
\draw[line width=1pt, gray!30] (3) -- (9);
\draw[line width=1pt, gray!30] (3) -- (10);
\draw[line width=1pt, gray!30] (3) -- (11);
\draw[line width=1pt, gray!30] (3) -- (12);
\draw[line width=1pt, gray!30] (3) -- (13);

\draw[line width=1pt, gray!30] (4) -- (8);
\draw[line width=1pt, gray!30] (4) -- (9);
\draw[line width=1pt, gray!30] (4) -- (10);
\draw[line width=1pt, gray!30] (4) -- (11);
\draw[line width=1pt, gray!30] (4) -- (12);
\draw[line width=1pt, gray!30] (4) -- (13);

\draw[line width=1pt, gray!30] (5) -- (8);
\draw[line width=1pt, gray!30] (5) -- (9);
\draw[line width=1pt, gray!30] (5) -- (10);
\draw[line width=1pt, gray!30] (5) -- (11);
\draw[line width=1pt, gray!30] (5) -- (12);
\draw[line width=1pt, gray!30] (5) -- (13);

\draw[line width=1pt, gray!30] (6) -- (8);
\draw[line width=1pt, gray!30] (6) -- (9);
\draw[line width=1pt, gray!30] (6) -- (10);
\draw[line width=1pt, gray!30] (6) -- (11);
\draw[line width=1pt, gray!30] (6) -- (12);
\draw[line width=1pt, gray!30] (6) -- (13);

\draw[line width=1pt, gray!30] (7) -- (8);
\draw[line width=1pt, gray!30] (7) -- (9);
\draw[line width=1pt, gray!30] (7) -- (10);
\draw[line width=1pt, gray!30] (7) -- (11);
\draw[line width=1pt, gray!30] (7) -- (12);
\draw[line width=1pt, gray!30] (7) -- (13);

\draw[line width=1pt, gray!30] (8) -- (9);
\draw[line width=1pt, gray!30] (8) -- (10);
\draw[line width=1pt, gray!30] (8) -- (11);
\draw[line width=1pt, gray!30] (8) -- (12);
\draw[line width=1pt, gray!30] (8) -- (13);

\draw[line width=1pt, gray!30] (9) -- (10);
\draw[line width=1pt, gray!30] (9) -- (11);
\draw[line width=1pt, gray!30] (9) -- (12);
\draw[line width=1pt, gray!30] (9) -- (13);

\draw[line width=1pt, gray!30] (10) -- (11);
\draw[line width=1pt, gray!30] (10) -- (12);
\draw[line width=1pt, gray!30] (10) -- (13);

\draw[line width=1pt, gray!30] (11) -- (12);
\draw[line width=1pt, gray!30] (11) -- (13);

\draw[line width=1pt, gray!30] (12) -- (13);

\draw[line width=1pt, red] (1) -- (2);
\draw[line width=1pt, red] (1) -- (3);
\draw[line width=1pt, red] (1) -- (4);
\draw[line width=1pt, blue] (1) -- (5);
\draw[line width=1pt, blue] (1) -- (6);
\draw[line width=1pt, blue] (1) -- (7);

\draw[line width=1pt, red] (2) -- (3);
\draw[line width=1pt, red] (2) -- (4);
\draw[line width=1pt, green] (2) -- (5);
\draw[line width=1pt, green] (2) -- (6);
\draw[line width=1pt, green] (2) -- (7);

\draw[line width=1pt, red] (3) -- (4);
\draw[line width=1pt, green] (3) -- (5);
\draw[line width=1pt, green] (3) -- (6);
\draw[line width=1pt, green] (3) -- (7);

\draw[line width=1pt, green] (4) -- (5);
\draw[line width=1pt, green] (4) -- (6);
\draw[line width=1pt, green] (4) -- (7);

\draw[line width=1pt, blue] (5) -- (6);
\draw[line width=1pt, blue] (5) -- (7);

\draw[line width=1pt, blue] (6) -- (7);

\begin{scope}
    \clip (-3.5,0) rectangle (3.5,3.5);
    \draw (0,0) circle(3.5);
    \draw (-3.5,0) -- (3.5,0);
\end{scope}

\end{tikzpicture} &&
\begin{tikzpicture}
[scale=0.6, every node/.style={fill=black, circle, font=\scriptsize, very thin, scale=.4}]
\drawlabels[13]
\draw[line width=1pt, psychedelicpurple] (1) -- (8);
\draw[line width=1pt, psychedelicpurple] (1) -- (9);
\draw[line width=1pt, psychedelicpurple] (1) -- (10);
\draw[line width=1pt, uclagold] (1) -- (11);
\draw[line width=1pt, uclagold] (1) -- (12);
\draw[line width=1pt, uclagold] (1) -- (13);

\draw[line width=1pt, psychedelicpurple] (2) -- (8);
\draw[line width=1pt, psychedelicpurple] (2) -- (9);
\draw[line width=1pt, psychedelicpurple] (2) -- (10);
\draw[line width=1pt, uclagold] (2) -- (11);
\draw[line width=1pt, uclagold] (2) -- (12);
\draw[line width=1pt, uclagold] (2) -- (13);

\draw[line width=1pt, uclagold] (3) -- (8);
\draw[line width=1pt, uclagold] (3) -- (9);
\draw[line width=1pt, uclagold] (3) -- (10);
\draw[line width=1pt, psychedelicpurple] (3) -- (11);
\draw[line width=1pt, psychedelicpurple] (3) -- (12);
\draw[line width=1pt, psychedelicpurple] (3) -- (13);

\draw[line width=1pt, psychedelicpurple] (4) -- (8);
\draw[line width=1pt, psychedelicpurple] (4) -- (9);
\draw[line width=1pt, psychedelicpurple] (4) -- (10);
\draw[line width=1pt, uclagold] (4) -- (11);
\draw[line width=1pt, uclagold] (4) -- (12);
\draw[line width=1pt, uclagold] (4) -- (13);

\draw[line width=1pt, uclagold] (5) -- (8);
\draw[line width=1pt, uclagold] (5) -- (9);
\draw[line width=1pt, uclagold] (5) -- (10);
\draw[line width=1pt, red] (5) -- (11);
\draw[line width=1pt, red] (5) -- (12);
\draw[line width=1pt, red] (5) -- (13);

\draw[line width=1pt, red] (6) -- (8);
\draw[line width=1pt, red] (6) -- (9);
\draw[line width=1pt, red] (6) -- (10);
\draw[line width=1pt, psychedelicpurple] (6) -- (11);
\draw[line width=1pt, psychedelicpurple] (6) -- (12);
\draw[line width=1pt, psychedelicpurple] (6) -- (13);

\draw[line width=1pt, uclagold] (7) -- (8);
\draw[line width=1pt, uclagold] (7) -- (9);
\draw[line width=1pt, uclagold] (7) -- (10);
\draw[line width=1pt, psychedelicpurple] (7) -- (11);
\draw[line width=1pt, psychedelicpurple] (7) -- (12);
\draw[line width=1pt, psychedelicpurple] (7) -- (13);

\draw[line width=1pt, red] (8) -- (9);
\draw[line width=1pt, red] (8) -- (10);
\draw[line width=1pt, blue] (8) -- (11);
\draw[line width=1pt, blue] (8) -- (12);
\draw[line width=1pt, blue] (8) -- (13);

\draw[line width=1pt, red] (9) -- (10);
\draw[line width=1pt, blue] (9) -- (11);
\draw[line width=1pt, blue] (9) -- (12);
\draw[line width=1pt, blue] (9) -- (13);

\draw[line width=1pt, blue] (10) -- (11);
\draw[line width=1pt, blue] (10) -- (12);
\draw[line width=1pt, blue] (10) -- (13);

\draw[line width=1pt, red] (11) -- (12);
\draw[line width=1pt, red] (11) -- (13);

\draw[line width=1pt, red] (12) -- (13);


\draw[line width=1pt, green] (2) -- (5);
\draw[line width=1pt, green] (2) -- (6);
\draw[line width=1pt, green] (2) -- (7);

\draw[line width=1pt, green] (3) -- (5);
\draw[line width=1pt, green] (3) -- (6);
\draw[line width=1pt, green] (3) -- (7);

\draw[line width=1.pt, green] (4) -- (5);
\draw[line width=1pt, green] (4) -- (6);
\draw[line width=1pt, green] (4) -- (7);

\draw[line width=1pt, blue] (1) -- (5);
\draw[line width=1pt, blue] (1) -- (6);
\draw[line width=1pt, blue] (1) -- (7);

\draw[line width=1pt, blue] (5) -- (6);
\draw[line width=1pt, blue] (5) -- (7);

\draw[line width=1pt, blue] (6) -- (7);

\draw[line width=1pt, red] (1) -- (2);
\draw[line width=1pt, red] (1) -- (3);
\draw[line width=1pt, red] (1) -- (4);

\draw[line width=1pt, red] (2) -- (3);
\draw[line width=1pt, red] (2) -- (4);

\draw[line width=1pt, red] (3) -- (4);

\begin{scope}
    \clip (-3.5,0) rectangle (3.5,3.5);
    \draw (0,0) circle(3.5);
    \draw (-3.5,0) -- (3.5,0);
\end{scope}
\end{tikzpicture}\\
(a) $K_7$ with $W_1$ & & (b) $G$\\
\end{tabular}
\begin{tabular}{c}
\begin{tikzpicture}[scale=0.35, every node/.style={fill=black, circle, font=\scriptsize, very thin, scale=.35}]
\clip (-5,6) rectangle + (28,-11.5);
    \drawlabels[7]

    \draw[line width=1pt, red] (1) -- (2);
    \draw[line width=1pt, red] (1) -- (3);
    \draw[line width=1pt, red] (1) -- (4);
    \draw[line width=1pt, blue] (1) -- (5);
    \draw[line width=1pt, blue] (1) -- (6);
    \draw[line width=1pt, blue] (1) -- (7);

    \draw[line width=1pt, red] (2) -- (3);
    \draw[line width=1pt, red] (2) -- (4);
    \draw[line width=1pt, green] (2) -- (5);
    \draw[line width=1pt, green] (2) -- (6);
    \draw[line width=1pt, green] (2) -- (7);

    \draw[line width=1pt, red] (3) -- (4);
    \draw[line width=1pt, green] (3) -- (5);
    \draw[line width=1pt, green] (3) -- (6);
    \draw[line width=1pt, green] (3) -- (7);

    \draw[line width=1pt, green] (4) -- (5);
    \draw[line width=1pt, green] (4) -- (6);
    \draw[line width=1pt, green] (4) -- (7);

    \draw[line width=1pt, blue] (5) -- (6);
    \draw[line width=1pt, blue] (5) -- (7);

    \draw[line width=1pt, blue] (6) -- (7);

    \node[circle,draw,minimum size=\rightvertexsize,fill=black] (8) at (\rightvertexdist,0) {};

    \draw[line width=2.25pt, bend left=1, gray!30] (1) to (8);
    \draw[line width=2.25pt, bend right=11, gray!30] (1) to (8);

    \draw[line width=2.25pt, bend left=7, gray!30] (2) to (8);
    \draw[line width=2.25pt, bend left=20, gray!30] (2) to (8);

    \draw[line width=2.25pt, bend left=50, gray!30] (3) to (8);
    \draw[line width=2.25pt, bend left=65, gray!30] (3) to (8);

    \draw[line width=2.25pt, bend right=95, gray!30] (4) to (8);
    \draw[line width=2.25pt, out=-110, in=-70, distance=6.2cm, gray!30] (4) to (8);

    \draw[line width=2.25pt, bend right=32, gray!30] (5) to (8);
    \draw[line width=2.25pt, bend right=42, gray!30] (5) to (8);

    \draw[line width=2.25pt, bend right=5, gray!30] (6) to (8);
    \draw[line width=2.25pt, bend right=15, gray!30] (6) to (8);

    \draw[line width=2.25pt, bend left=5, gray!30] (7) to (8);
    \draw[line width=2.25pt, bend right=10, gray!30] (7) to (8);

    \draw[line width=2.25pt, gray!30] (8) to[out=-22, in=22, looseness=15] (8);
    \draw[line width=2.25pt, gray!30] (8) to[out=-35, in=35, looseness=17] (8);
    \draw[line width=2.25pt, gray!30] (8) to[out=-42, in=42, looseness=18] (8);
    \draw[line width=2.25pt, gray!30] (8) to[out=-50, in=50, looseness=21] (8);
    \draw[line width=2.25pt, gray!30] (8) to[out=-57, in=57, looseness=25] (8);

    \begin{scope}[shift={(\rightvertexdist--7,0)}, every node/.style={fill=none, rectangle, font=\scriptsize}]
    \draw[thin, black] (0.8,0.55) rectangle (6,-1.25); 
    \draw[line width=1pt, black] (1,0) -- (2,0) node[right] {= 1 edge};
    \draw[line width=2.25pt, black] (1,-0.75) -- (2,-0.75) node[right] {= 3 edges};
\end{scope}
\end{tikzpicture}\\
(c) $G_1$\\
\begin{tikzpicture}[scale=0.35, every node/.style={fill=black, circle, font=\scriptsize, very thin, scale=.35}]
\clip (-5,6) rectangle + (28,-11.5);
    \drawlabels[7]

    \draw[line width=1pt, red] (1) -- (2);
    \draw[line width=1pt, red] (1) -- (3);
    \draw[line width=1pt, red] (1) -- (4);
    \draw[line width=1pt, blue] (1) -- (5);
    \draw[line width=1pt, blue] (1) -- (6);
    \draw[line width=1pt, blue] (1) -- (7);

    \draw[line width=1pt, red] (2) -- (3);
    \draw[line width=1pt, red] (2) -- (4);
    \draw[line width=1pt, green] (2) -- (5);
    \draw[line width=1pt, green] (2) -- (6);
    \draw[line width=1pt, green] (2) -- (7);

    \draw[line width=1pt, red] (3) -- (4);
    \draw[line width=1pt, green] (3) -- (5);
    \draw[line width=1pt, green] (3) -- (6);
    \draw[line width=1pt, green] (3) -- (7);

    \draw[line width=1pt, green] (4) -- (5);
    \draw[line width=1pt, green] (4) -- (6);
    \draw[line width=1pt, green] (4) -- (7);

    \draw[line width=1pt, blue] (5) -- (6);
    \draw[line width=1pt, blue] (5) -- (7);

    \draw[line width=1pt, blue] (6) -- (7);

    \node[circle,draw,minimum size=\rightvertexsize,fill=black] (8) at (\rightvertexdist,0) {};

    \draw[line width=2.25pt, bend left=1, psychedelicpurple] (1) to (8);
    \draw[line width=2.25pt, bend right=11, uclagold] (1) to (8);

    \draw[line width=2.25pt, bend left=7, uclagold] (2) to (8);
    \draw[line width=2.25pt, bend left=20, psychedelicpurple] (2) to (8);

    \draw[line width=2.25pt, bend left=50, psychedelicpurple!80] (3) to (8);
    \draw[line width=2.25pt, bend left=65, uclagold] (3) to (8);

    \draw[line width=2.25pt, bend right=95, psychedelicpurple] (4) to (8);
    \draw[line width=2.25pt, out=-110, in=-70, distance=6.2cm, uclagold] (4) to (8);

    \draw[line width=2.25pt, bend right=32, psychedelicpurple] (5) to (8);
    \draw[line width=2.25pt, bend right=42, uclagold] (5) to (8);

    \draw[line width=2.25pt, bend right=5, red] (6) to (8);
    \draw[line width=2.25pt, bend right=15, uclagold] (6) to (8);

    \draw[line width=2.25pt, bend left=5, psychedelicpurple] (7) to (8);
    \draw[line width=2.25pt, bend right=10, red] (7) to (8);

    \draw[line width=2.25pt, gray!30] (8) to[out=-22, in=22, looseness=15] (8);
    \draw[line width=2.25pt, gray!30] (8) to[out=-35, in=35, looseness=17] (8);
    \draw[line width=2.25pt, gray!30] (8) to[out=-42, in=42, looseness=18] (8);
    \draw[line width=2.25pt, gray!30] (8) to[out=-50, in=50, looseness=21] (8);
    \draw[line width=2.25pt, gray!30] (8) to[out=-57, in=57, looseness=25] (8);

    \begin{scope}[shift={(\rightvertexdist--7,0)}, every node/.style={fill=none, rectangle, font=\scriptsize}]
    \draw[thin, black] (0.8,0.55) rectangle (6,-1.25); 
    \draw[line width=1pt, black] (1,0) -- (2,0) node[right] {= 1 edge};
    \draw[line width=2.25pt, black] (1,-0.75) -- (2,-0.75) node[right] {= 3 edges};
\end{scope}
\end{tikzpicture}\\
\\
(d) $G_1 \cup G_2$\\
\begin{tikzpicture}[scale=0.35, every node/.style={fill=black, circle, font=\scriptsize, very thin, scale=.35}]
\clip (-5,6) rectangle + (28,-11.5);
    \drawlabels[7]

    \draw[line width=1pt, red] (1) -- (2);
    \draw[line width=1pt, red] (1) -- (3);
    \draw[line width=1pt, red] (1) -- (4);
    \draw[line width=1pt, blue] (1) -- (5);
    \draw[line width=1pt, blue] (1) -- (6);
    \draw[line width=1pt, blue] (1) -- (7);

    \draw[line width=1pt, red] (2) -- (3);
    \draw[line width=1pt, red] (2) -- (4);
    \draw[line width=1pt, green] (2) -- (5);
    \draw[line width=1pt, green] (2) -- (6);
    \draw[line width=1pt, green] (2) -- (7);

    \draw[line width=1pt, red] (3) -- (4);
    \draw[line width=1pt, green] (3) -- (5);
    \draw[line width=1pt, green] (3) -- (6);
    \draw[line width=1pt, green] (3) -- (7);

    \draw[line width=1pt, green] (4) -- (5);
    \draw[line width=1pt, green] (4) -- (6);
    \draw[line width=1pt, green] (4) -- (7);

    \draw[line width=1pt, blue] (5) -- (6);
    \draw[line width=1pt, blue] (5) -- (7);

    \draw[line width=1pt, blue] (6) -- (7);

    \node[circle,draw,minimum size=\rightvertexsize,fill=black] (8) at (\rightvertexdist,0) {};

    \draw[line width=2.25pt, bend left=1, psychedelicpurple] (1) to (8);
    \draw[line width=2.25pt, bend right=11, uclagold] (1) to (8);

    \draw[line width=2.25pt, bend left=7, uclagold] (2) to (8);
    \draw[line width=2.25pt, bend left=20, psychedelicpurple] (2) to (8);

    \draw[line width=2.25pt, bend left=50, psychedelicpurple!80] (3) to (8);
    \draw[line width=2.25pt, bend left=65, uclagold] (3) to (8);

    \draw[line width=2.25pt, bend right=95, psychedelicpurple] (4) to (8);
    \draw[line width=2.25pt, out=-110, in=-70, distance=6.2cm, uclagold] (4) to (8);

    \draw[line width=2.25pt, bend right=32, psychedelicpurple] (5) to (8);
    \draw[line width=2.25pt, bend right=42, uclagold] (5) to (8);

    \draw[line width=2.25pt, bend right=5, red] (6) to (8);
    \draw[line width=2.25pt, bend right=15, uclagold] (6) to (8);

    \draw[line width=2.25pt, bend left=5, psychedelicpurple] (7) to (8);
    \draw[line width=2.25pt, bend right=10, red] (7) to (8);

    \draw[line width=2.25pt, red] (8) to[out=-22, in=22, looseness=15] (8);
    \draw[line width=2.25pt, red] (8) to[out=-35, in=35, looseness=17] (8);
    \draw[line width=2.25pt, blue] (8) to[out=-42, in=42, looseness=18] (8);
    \draw[line width=2.25pt, blue] (8) to[out=-50, in=50, looseness=21] (8);
    \draw[line width=2.25pt, blue] (8) to[out=-57, in=57, looseness=25] (8);

    \begin{scope}[shift={(\rightvertexdist--7,0)}, every node/.style={fill=none, rectangle, font=\scriptsize}]
    \draw[thin, black] (0.8,0.55) rectangle (6,-1.25); 
    \draw[line width=1pt, black] (1,0) -- (2,0) node[right] {= 1 edge};
    \draw[line width=2.25pt, black] (1,-0.75) -- (2,-0.75) node[right] {= 3 edges};
\end{scope}
\end{tikzpicture}\\
(e) $G'=G_1 \cup G_2 \cup G_3$\\
\end{tabular}
\caption{An embedding of an equitable $(3,2)$-edge-coloring of $K_7$ in an equitable $(5,4)$-edge-coloring of $K_{13}$.}
\label{expic}
\end{figure}
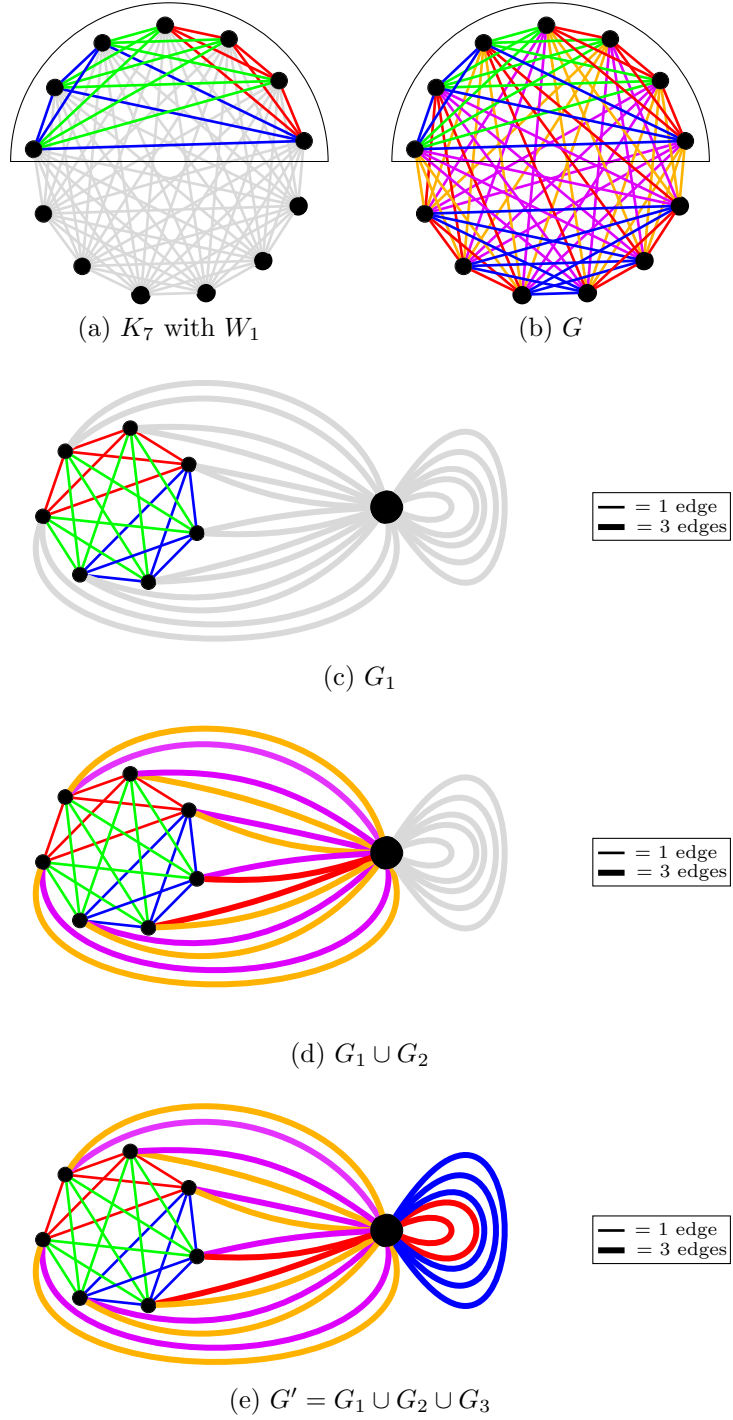

\subsection{Sufficiency when $2t_1 < t_2 < 4t_1$}
\label{subsectt2>2t1}

We begin with a few results that will be used throughout this section.

The following two lemmas can be proven with a Walecki-type construction.
A \emph{decomposition} of a graph $G$ into copies of $H$ is a partition of the edges of $G$ into copies of $H$. 

\begin{lemma} \cite{mine} \label{Kdecomp}
$2K_n$ can be decomposed into $n$ cycles of length $n-1$.
\end{lemma}

A \emph{$k$-factor} of a graph $G$ is a spanning subgraph of $G$ in which every vertex has degree $k$.
Note that a $1$-factor is a \emph{perfect matching}.
A \emph{$k$-factorization} is a partition of the edges of $G$ into disjoint $k$-factors.
The following is a well-known result that is a special case of Baranyai's Theorem \cite{baranyai1974factrization} and was known in the 19th century.

\begin{lemma} \label{1factKn}
If $n$ is even, then there is a $1$-factorization of $K_n$.
\end{lemma}

The following theorem is attributed to Petersen and can be proven using Euler circuits and $1$-factorizations of bipartite graphs.

\begin{theorem} \label{petersen}
Every $2k$-regular graph has a $2$-factorization (into $k$ edge-disjoint $2$-factors).
\end{theorem}

We now consider the case where  $2t_1 < t_2 < 4t_1$, and will again use the proof technique of amalgamations.
Although we believe that Lemma \ref{obs:New-color-method-theorem} can be modified to remove the condition that $t_2<4t_1$, this significantly simplified the proof and is all that is needed in Theorem \ref{mainthm}.

\begin{lemma} \label{obs:New-color-method-theorem}
Let $2t_1 <t_2 < 4t_1$.
An equitable $(2t_1+1, 2t_1)$-edge-coloring of $K_{n_1}$ can be embedded in an equitable $(2t_2+1, 2t_2)$-edge-coloring of $K_{n_2}$.
\end{lemma}

\begin{proof}
Recall that $m= \vert C_N \vert = 2t_2-2t_1$.
Let $X = \{x_1, x_2, \dots x_{n_1}\}$ and $Y= \{y_{0}, y_{1}, \dots y_{m-1}\}$.
We will form edge-colored graphs $G_1, G_2, \dots G_6$ on the vertex set $X \cup Y$ as follows.
Once this is done defining $G'= \cup_{i=1}^6 G_i$, Theorem \ref{amalgamationtheorem} will be used to detach $G'$ to form $G= K_{n_2}$ with our desired equitable $(2t_2+1,2t_2)$-edge-coloring and an arbitrary $(2t_1+1, 2t_1)$-edge-coloring of $K_{n_1}$ as a subgraph.
To do so, $\phi: V(G) \rightarrow V(G')$ will be the amalgamation function with amalgamation numbers $\eta(x_i)= 1$ for $1 \leq i \leq n_1$ and $\eta(y_p)= b$ for $ 0 \leq p \leq  m-1$.

We consider $2$ cases based on the parity of $a$ in the following equation.
These cases first differ when defining $G_3$ in (iii).
By the division algorithm there exists a unique $a$ and $r$ such that $0 \leq r \leq m-2$ and \begin{equation} \label{divalg}
n_1b = a (m-1) +r.
\end{equation}
Also, as we will need in (vi) in both cases,
\begin{equation} \label{a>2}
a > 2
\end{equation}
because
$$\begin{array}{lllll} 
a&=& \frac{n_1b-r}{m-1}&\phantom{W}&\text{By (\ref{divalg}})\\[3pt]
&>& \frac{n_1b}{m-1}-1& &\text{Because }r\leq m-2\\[3pt]
&=& \frac{(2t_1b+1)b}{2t_2-2t_1-1}-1\\[3pt]
&>& \frac{(2t_1b+1)b}{6t_1-1}-1& &\text{Because } t_2< 4t_1\\[3pt]
&\geq& \frac{(2t_1(3)+1)(3)}{6t_1-1}-1& &\text{Because } b\geq 3\\[3pt]
&=& (\frac{6t_1+1}{6t_1-1})3-1\\[3pt]
&>& (1)3-1\\[3pt]
&=& 2.
\end{array}$$
For clarity, the cases are written completely separately with notes where the two cases are the same.

\noindent
{\bf Case 1:} Assume $a$ is even.

\begin{enumerate}[i)]
\item Let $G_1=K_{n_1}$ on the vertex set $X$ with any $(2t_1+1, 2t_1)$-edge-coloring, call it $W_1$, using the colors in $C_O$ (such a coloring exists by Theorem \ref{mainlambda1}). Let $S_c$ be the set of vertices in $X$ missing color $c$ in $W_1$ for $1 \leq c \leq 2t_1+1$. 

\item Form $G_{2}=bK_{n_1,m}$ with parts $X$ and $Y$ using the following steps. For $y_p\in Y$, let color $s_1+1+p$ be missing from $y_p$.
For $x_i\in X$, we will color $b$ edges with color $s_1+1+p$ between $x_i$ and $Y$.
Informally, for $0 \leq p \leq m-1$, first color the edge $\{x_1, y_{p+1 \Mod{m}}\}$ with color $s_1+1+p$ and then color the edge between $x_1$ and ``next" vertex in $Y$ skipping $y_p$ until there are $b$ edges of color $s_1+1+p$ between $x_1$ and $Y$.
Then color $b$ edges between $x_2$ and $Y$ by ``Picking up where we left off in $Y$".
Continue to do this until each vertex in $X$ has $b$ edges colored $s_1+1+p$ between it and $Y$.
Formally, for $x_i\in X$, color an edge $\{x_i,y_{p+h+1\Mod{m}}\}$ with color $s_1+1+p$, where $h \equiv (j+(i-1)b)\Mod{m-1}$ 
for $0 \leq p \leq m-1$ and $0\leq j\leq b-1$
(see Figure \ref{bipartex} (a)).
Note that each vertex $x_i \in X$ is incident to $b$ edges of color $s_1+1+p$, the vertices $y_{p+1 \Mod{m}}, y_{p+2\Mod{m}}, \dots ,y_{p+r \Mod{m}}$ are incident to $a+1$ edges colored $s_1+1+p$, the vertices $y_{p+r+1 \Mod{m}}, y_{p+r+2 \Mod{m}}, \dots, y_{p-1 \Mod{m}}$ are incident to $a$ edges colored $s_1+1+p$, and $y_p$ is incident to no edges colored $s_1+1+p$ (see Figure \ref{bipartex} (b)).

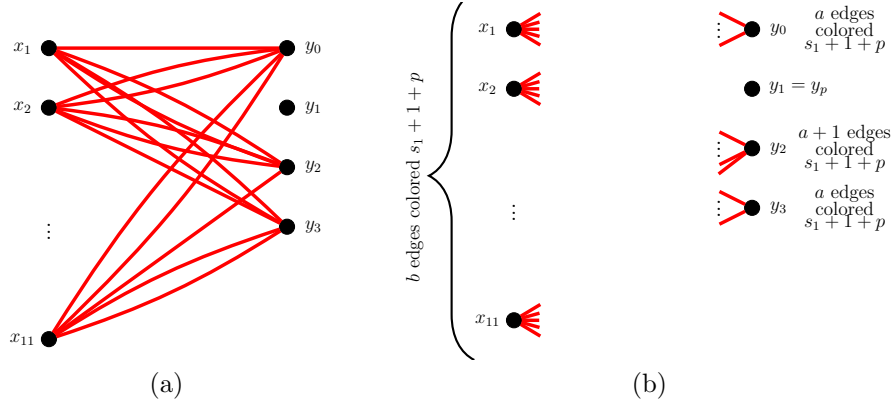
\begin{figure}
\centering
\begin{tabular}{ccc}
    \begin{tikzpicture}
[scale=0.35, every node/.style={draw,shape=circle,fill=black, scale=0.6}]
\node[black] at (-4.5, 8)   (1) {};
\node[black] at (-4.5, 5.75)  (2) {};
\node[draw=none, fill=none, align=center] at (-4.5, 1.25) {\Large $\vdots$};
\node[black] at (-4.5, -3)   (3) {};
\node[black] at (4.5, 8)   (4) {};
\node[black] at (4.5, 5.75)  (5) {};
\node[black] at (4.5, 3.5)   (6) {};
\node[black] at (4.5, 1.25)   (7) {};

\draw[line width=1.3pt, bend left=7, color=red] (1) to (7);
\draw[line width=1.3pt, bend right=7, color=red] (1) to (7);

\draw[line width=1.3pt, bend left=7, color=red] (1) to (6);
\draw[line width=1.3pt, bend right=7, color=red] (1) to (6);

\draw[line width=1.3pt, color=red] (1) to (4);

\draw[line width=1.3pt, bend left=7, color=red] (2) to (6);
\draw[line width=1.3pt, bend right=7, color=red] (2) to (6);

\draw[line width=1.3pt, bend left=7, color=red] (2) to (4);
\draw[line width=1.3pt, bend right=7, color=red] (2) to (4);

\draw[line width=1.3pt, color=red] (2) to (7);

\draw[line width=1.3pt, bend left=7, color=red] (3) to (4);
\draw[line width=1.3pt, bend right=7, color=red] (3) to (4);

\draw[line width=1.3pt, bend left=7, color=red] (3) to (7);
\draw[line width=1.3pt, bend right=7, color=red] (3) to (7);

\draw[line width=1.3pt, color=red] (3) to (6);

\node[draw=none, fill = none, align=center] at (5.5,8) {\large $y_0$};
\node[draw=none, fill = none, align=center] at (5.5,5.75) {\large $y_1$};
\node[draw=none, fill = none, align=center] at (5.5,3.5) {\large $y_2$};
\node[draw=none, fill = none, align=center] at (5.5,1.25) {\large $y_3$};

\node[draw=none, fill = none, align=center] at (-5.5,8) {\large $x_1$};
\node[draw=none, fill = none, align=center] at (-5.5,5.75) { \large $x_2$};
\node[draw=none, fill = none, align=center] at (-5.5,-3) {\large $x_{11}$};

\end{tikzpicture} &
& \begin{tikzpicture}
[scale=0.35, every node/.style={draw,shape=circle,fill=black, scale=0.6}]
    \clip (-9,9) rectangle (10.25,-4.5);

\node[black] at (-4.5, 8)   (1) {};
\node[black] at (-4.5, 5.75)  (2) {};
\node[draw=none, fill=none, align=center] at (-4.5, 1.25) {\Large $\vdots$};
\node[black] at (-4.5, -3)   (3) {};
\node[black] at (4.5, 8)   (4) {};
\node[black] at (4.5, 5.75)  (5) {};
\node[black] at (4.5, 3.5)   (6) {};
\node[black] at (4.5, 1.25)   (7) {};

\node[white] at (-3.25, 8.75)   (101) {};
\node[white] at (-3.25, 8.35)   (102) {};
\node[white] at (-3.25, 8)   (103) {};
\node[white] at (-3.25, 7.65)   (104) {};
\node[white] at (-3.25, 7.25)   (105) {};

\node[white] at (-3.25, 6.5)   (201) {};
\node[white] at (-3.25, 6.1)   (202) {};
\node[white] at (-3.25, 5.75)   (203) {};
\node[white] at (-3.25, 5.4)   (204) {};
\node[white] at (-3.25, 5)   (205) {};

\node[white] at (-3.25, -2.25)   (301) {};
\node[white] at (-3.25, -2.65)   (302) {};
\node[white] at (-3.25, -3)   (303) {};
\node[white] at (-3.25, -3.35)   (304) {};
\node[white] at (-3.25, -3.75)   (305) {};

\node[white] at (3, 8.75)   (401) {};
\node[draw=none, fill=none, align=center] at (3.25, 8.15)  {\Large $\vdots$};
\node[white] at (3, 7.25)   (403) {};

\node[white] at (3, 4.25)   (601) {};
\node[draw=none, fill=none, align=center] at (3.25, 3.65)  {\Large $\vdots$};
\node[white] at (3, 2.75)   (603) {};
\node[white] at (3, 2.35)   (604) {};

\node[white] at (3, 2)   (701) {};
\node[draw=none, fill=none, align=center] at (3.25, 1.4)  {\Large $\vdots$};
\node[white] at (3, 0.5)   (703) {};

\draw[line width=1.3pt, color=red] (1) to (101);
\draw[line width=1.3pt, color=red] (1) to (102);
\draw[line width=1.3pt, color=red] (1) to (103);
\draw[line width=1.3pt, color=red] (1) to (104);
\draw[line width=1.3pt, color=red] (1) to (105);

\draw[line width=1.3pt, color=red] (2) to (201);
\draw[line width=1.3pt, color=red] (2) to (202);
\draw[line width=1.3pt, color=red] (2) to (203);
\draw[line width=1.3pt, color=red] (2) to (204);
\draw[line width=1.3pt, color=red] (2) to (205);

\draw[line width=1.3pt, color=red] (3) to (301);
\draw[line width=1.3pt, color=red] (3) to (302);
\draw[line width=1.3pt, color=red] (3) to (303);
\draw[line width=1.3pt, color=red] (3) to (304);
\draw[line width=1.3pt, color=red] (3) to (305);

\draw[line width=1.3pt, color=red] (4) to (401);
\draw[line width=1.3pt, color=red] (4) to (403);

\draw[line width=1.3pt, color=red] (6) to (601);
\draw[line width=1.3pt, color=red] (6) to (603);
\draw[line width=1.3pt, color=red] (6) to (604);

\draw[line width=1.3pt, color=red] (7) to (701);
\draw[line width=1.3pt, color=red] (7) to (703);

\node[draw=none, fill = none, align=center] at (5.5,8) {\large $y_0$};
\node[draw=none, fill = none, align=center] at (6.25,5.75) {\large $y_1=y_p$};
\node[draw=none, fill = none, align=center] at (5.5,3.5) {\large $y_2$};
\node[draw=none, fill = none, align=center] at (5.5,1.25) {\large $y_3$};

\node[draw=none, fill = none, align=center] at (-5.5,8) {\large $x_1$};
\node[draw=none, fill = none, align=center] at (-5.5,5.75) {\large $x_2$};
\node[draw=none, fill = none, align=center] at (-5.5,-3) {\large $x_{11}$};


\node[draw=none, fill = none, align=center] at (-8.25, 2.4) {\rotatebox{90}{\large $b$ edges colored $s_1+1+p$}};

\node[draw=none, fill = none, align=center] at (8,8) {\large $a$ edges\\ \large colored \\ \large $s_1+1+p$};
\node[draw=none, fill = none, align=center] at (8,3.5) {\large $a+1$ edges\\ \large colored \\ \large $s_1+1+p$};
\node[draw=none, fill = none, align=center] at (8,1.25) {\large $a$ edges\\ \large colored \\ \large $s_1+1+p$};

\textbf{\draw[thick, decorate, decoration={brace, amplitude=0.6cm, aspect=0.5}, xscale=-1, shift={(9,-7.5)}] 
    plot [smooth, tension=3.25] coordinates {(-3,3) (-3,3) (0,1) (2,2) (3,3)};}

\end{tikzpicture}\\
(a)&&(b)\\
\end{tabular}  
\caption{Two representations of the graph induced by the edges colored $s_1+2$ in $G_2$ for the embedding of an equitable $(3,2)$-edge-coloring of $K_{11}$ into an equitable $(7,6)$-edge-coloring of $K_{31}$. Here, $n_1b=11(5)=18(3)+1=a(m-1)+r$. Therefore, $r=1$ vertices in $Y$ are incident to $a+1=19$ edges colored $s_1+2$, $m-r-1=2$ vertices in $Y$ are incident to $a=18$ edges colored $s_1+2$, and $1$ vertex in $Y$ is incident to no edges colored $s_1+2$.}
\label{bipartex}
\end{figure}

We now check the coloring on $G_2$ is well-defined by checking that $m_{G_2}(x_i,y_p)\\ =b$.
By the division algorithm, there exists unique nonnegative integers $\alpha$ and $\gamma$ such that $0 \leq \gamma \leq m-2$ and $b=\alpha(m-1)+\gamma$.
(Note that the left side of this equation is the number of edges of a particular color between a vertex $x_i \in X$ and $Y$. Previously, in (\ref{divalg}), the left side of the equation was the number of edges of a particular color between $X$ and $Y$.)
For a given color $s_1+1+p$, the vertices $y_{p+1\Mod{m}},\dots,y_{p+\gamma\Mod{m}}$ are incident to $\alpha+1$ edges colored $s_1+1+p$, the vertices $y_{p+\gamma +1\Mod{m}},\dots,y_{p-1\Mod{m}}$ are incident to $\alpha$ edges colored $s_1+1+p$, and the vertex $y_p$ is incident to no edges colored $s_1+1+p$. 
Therefore, (considering ``for how many colors $y_p$ is in each of the previous $3$ groups") for $x_i\in X$ and $y_p\in Y$, there are exactly $\alpha+1$ edges of exactly $\gamma$ colors in $C_N$ between $x_i$ and $y_p$, exactly $\alpha$ edges of exactly $m-\gamma-1$ colors in $C_N$ between $x_i$ and $y_p$, and $0$ edges of exactly $1$ color in $C_N$ between $x_i$ and $y_p$. Then, for $x_i\in X$ and $y_p\in Y$, $m_{G_2}(x_i,y_p)=\gamma(\alpha+1)+(m-\gamma-1)(\alpha)+1(0)=\alpha(m-1)+\gamma=b$.

\item Form $B_p$ for $0\leq p\leq m-1$ by taking the disjoint union of $\frac{m-r-1}{2}$ copies of $K_2$ and $r+1$ copies of $K_1$ on the vertex set $Y$ using the following steps.
Attach vertices $y_{p+r+j\Mod{m}}$ and $y_{p-j\Mod{m}}$ with an edge colored $s_1+1+p$ for $1 \leq j \leq \frac{m-r-1}{2}$
(Since $m$ and $a$ are even, by (\ref{divalg}), $r$ is odd, so $m-r-1$ is even, and $m-r-1/2$ is an integer).
See Figure \ref{iiiex} (a).
Let $G_3 = \bigcup_{0\leq p\leq m-1} B_p$.
Note that $G_3$ is regular.
We claim it is $(m-r-1$)-regular.
By definition, $\vert E(B_p) \vert = \frac{m-r-1}{2}$.
Therefore, $\vert E(G_3) \vert = \frac{m(m-r-1)}{2}$.
So, $\sum_{y \in V(G_3)} \deg_{G_3}(y)= m(m-r-1)$.
Thus, $G_3$ is $(m-r-1)$-regular.

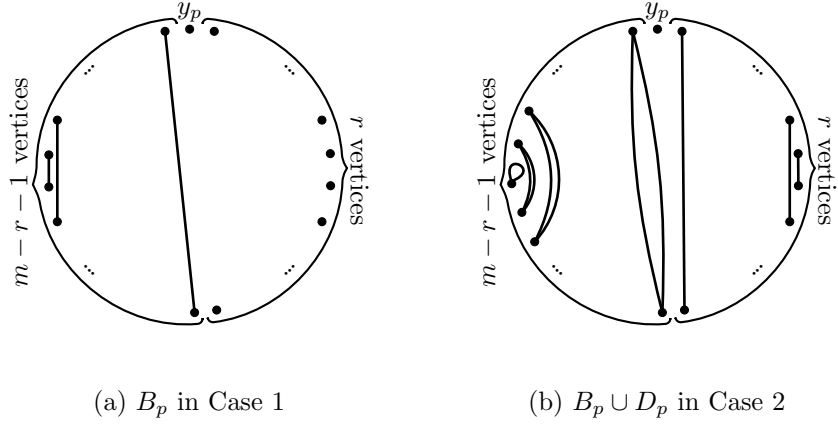
\begin{figure}
\centering
\begin{tabular}{ccc}
    \begin{tikzpicture}[scale=.75]
    
    \foreach \angle in {80, 90, 100, 159, 173.5, 186.5, 201, 272, 281, 339, 354, 367, 381} {
        \draw[fill] (\angle:2.5cm) circle (2pt);
    }
    
    \node at (43:2.5cm) {$\cdot$};
    \node at (45:2.5cm) {$\cdot$};
    \node at (47:2.5cm) {$\cdot$};
    \node at (133:2.5cm) {$\cdot$};
    \node at (135:2.5cm) {$\cdot$};
    \node at (137:2.5cm) {$\cdot$};
    \node at (223:2.5cm) {$\cdot$};
    \node at (225:2.5cm) {$\cdot$};
    \node at (227:2.5cm) {$\cdot$};
    \node at (313:2.5cm) {$\cdot$};
    \node at (315:2.5cm) {$\cdot$};
    \node at (317:2.5cm) {$\cdot$};
    
    \node at (0,2.8) {$y_p$};
    \node at (270:3.5cm) {};
    
    \curlybrace[tip angle=5,thick]{95}{275}{2.7}
    \curlybrace[tip angle=3,thick]{-83}{85}{2.7}

    \node at (-3, -0.2) {\rotatebox{90}{$m-r-1$ vertices}};
    \node at (3,0) {\rotatebox{270}{$r$ vertices}};
    

    \draw[line width=1pt, color=black] (159:2.5cm) to (201:2.5cm);
    \draw[line width=1pt, color=black] (173.5:2.5cm) to (186.5:2.5cm);
    \draw[line width=1pt, color=black] (100:2.5cm) to (272:2.5cm);
\end{tikzpicture}
&
\phantom{W}
&
\begin{tikzpicture}[scale=0.75]
    
    \foreach \angle in {80, 90, 100, 155, 169, 197, 210, 272, 281, 339, 354, 367, 381} {
        \draw[fill] (\angle:2.5cm) circle (2pt);
    }
    
    \draw[fill] (185:2.58cm) circle (2pt);
    \node at (43:2.5cm) {$\cdot$};
    \node at (45:2.5cm) {$\cdot$};
    \node at (47:2.5cm) {$\cdot$};
    \node at (133:2.5cm) {$\cdot$};
    \node at (135:2.5cm) {$\cdot$};
    \node at (137:2.5cm) {$\cdot$};
    \node at (223:2.5cm) {$\cdot$};
    \node at (225:2.5cm) {$\cdot$};
    \node at (227:2.5cm) {$\cdot$};
    \node at (313:2.5cm) {$\cdot$};
    \node at (315:2.5cm) {$\cdot$};
    \node at (317:2.5cm) {$\cdot$};
    
    \node at (0,2.8) {$y_p$};
    \node at (270:3.5cm) {};
    
    \curlybrace[tip angle=5,thick]{95}{275}{2.7}
    \curlybrace[tip angle=3,thick]{-83}{85}{2.7}

    \node at (-3, -0.2) {\rotatebox{90}{$m-r-1$ vertices}};
    \node at (3,0) {\rotatebox{270}{$r$ vertices}};
    
    
    \draw[line width=1pt, bend right=30, color=black] (210:2.5cm) to (155:2.5cm);
    \draw[line width=1pt, bend right=45, color=black] (210:2.5cm) to (155:2.5cm);
    \draw[line width=1pt, bend right=30, color=black] (197:2.5cm) to (169:2.5cm);
    \draw[line width=1pt, bend right=45, color=black] (197:2.5cm) to (169:2.5cm);
    \draw[line width=1pt, bend left=8, color=black] (100:2.5cm) to (272:2.5cm);
    \draw[line width=1pt, bend right=8, color=black] (100:2.5cm) to (272:2.5cm);

   \draw[line width=1pt, black] (184:2.58cm) .. controls +(20:0.6cm) and +(110:0.6cm) .. (184:2.58cm);

    \draw[line width=1pt, color=black] (339:2.5cm) to (381:2.5cm);
    \draw[line width=1pt, color=black] (354:2.5cm) to (367:2.5cm);
    \draw[line width=1pt, color=black] (80:2.5cm) to (281:2.5cm);

\end{tikzpicture}\\
(a) $B_p$ in Case 1 & & (b) $B_p \cup D_p$ in Case 2
\end{tabular}
\caption{Graphs defined in Case 1 (iii) and Case 2 (iii)}
\label{iiiex}
\end{figure}

\item Let $G_{4}=(b^2 - (a+1))K_m$ on the vertex set $Y$. Using Lemma \ref{Kdecomp}, decompose $G_{4}$ into $m\left(\frac{b^2-(a+1)}{2}\right)$ cycles of length $(m-1)$.
We color the edges of $G_4$ as follows.
For $0\leq p\leq m-1$, color the $\frac{b^2-(a+1)}{2}$ cycles with no edges incident to $y_p$ with color $s_1+1+p$ (This is possible because each cycle excludes exactly $1$ vertex $y_p \in Y$).
\item Using Lemma \ref{1factKn}, take a $1$-factorization of $K_m$ on the vertex set $Y$, producing $(m-1)$ edge-disjoint $1$-factors.
For each $k \in C_O$, create $F_{k}$  by coloring one distinct $1$-factor with color $k$.
This is possible because $m-1\geq 2t_1+1$
(Since $2t_1 <t_2$, $2t_1<2(t_2-t_1)=m$, and because $m$ is even, $m-1\geq 2t_1+1$).
Let $G_{5}=\bigcup_{k\in C_O}F_{k}$.
\item Let $H = b^2K_m$ on the vertex set $Y$ with $\frac{b(b-1)}{2}$ loops on each $y \in Y$.
Let $G_6 = H-(G_3 \cup G_4 \cup G_5)$.
We first check that $G_6$ exists (we have not removed nonexistent edges from $H$).
For each pair of vertices $y_p,y_s \in Y$ there are at most $2$ edges between $y_p$ and $y_s$ in $G_3$, exactly $b^2-(a+1)$ edges between $y_p$ and $y_s$ in $G_4$, and at most $1$ edge between $y_p$ and $y_s$ in $G_5$. The minimum number of edges between $y_p$ and $y_s$ in $G_6$ is $b^2-2-(b^2-(a+1))-1 = a-2$. By (\ref{a>2}), $a-2 > 0$, so $G_6$ exists.

The degree of each vertex in $H$ is $(b^2)(m-1)+b(b-1)$, in $G_3$ is $m-r-1$, in $G_4$ is $(b^2-(a+1))(m-1)$, and in $G_5$ is $2t_1+1$. Then, the degree of each vertex $y \in Y$ in $G_6$ is 

$\begin{array}{llll}
\deg_{G_6}(y) & = & \deg_{H}(y)-\deg_{G_3}(y)-\deg_{G_4}(y)-\deg_{G_5}(y) \\[4pt]
& = & b^2(m-1)+b(b-1)-(m-r-1)\\[4pt]
& & -(b^2-(a+1))(m-1))-(2t_1+1) \\[4pt]
&=& a(m-1)+r-(2t_1+1)+b(b-1)\\[4pt]
&=& n_1b-(2t_1+1)+b(b-1)& \text{By (\ref{divalg})}\\[4pt]
&=& (2t_1b+1)b-2t_1-1+b(b-1)\\[4pt]
&=& (2t_1+1)(b^2-1).\\
\end{array}$

By Theorem \ref{petersen}, $G_6$ can be partitioned into $\frac{(2t_1+1)(b^2-1)}{2}$ edge-disjoint $2$-factors. For $k\in C_O$, color $\frac{(b^2-1)}{2}$ distinct $2$-factors with color $k$.
\end{enumerate}

We define $G'=\bigcup_{i=1}^6G_i$. We will use Theorem \ref{amalgamationtheorem} to form the detachment $G=K_{n_2}$ with our desired equitable $(2t_2+1,2t_2)$-edge-coloring and an arbitrary equitable $(2t_1+1,2t_1)$-edge-coloring of $K_{n_1}$ as a subgraph. Recall that $\phi: V(G) \rightarrow V(G')$ is the amalgamation function with amalgamation numbers $\eta(x_i)= 1$ for $1 \leq i \leq n_1$ and $\eta(y_p)= b$ for $ 0 \leq p \leq  m-1$.

We first check that there are $b$ edges of each color incident to each vertex $v_i\in V_O$ except exactly one color in $C_O$. By (i), each $x_i\in X$ is missing some $c \in C_O$.

\begin{itemize}
\item For $k\in C_O\setminus \{c\}$, by (i), $deg_{G'(k)}(x_i)=b$. By Theorem \ref{amalgamationtheorem} (i), $deg_{G(k)}(v_i)=deg_{G'(k)}(x_i)/\eta(x_i)=b/1=b$ for $v_i\in \phi^{-1}(x_i)$.
\item For $k\in C_N$, by (ii), $deg_{G'(k)}(x_i)=b$. By Theorem \ref{amalgamationtheorem} (i), $deg_{G(k)}(v_i)=deg_{G'(k)}(x_i)/\eta(x_i)=b/1=b$ for $v_i\in \phi^{-1}(x_i)$.
\end{itemize}

We now check that each vertex $v_i\in V_N$ is incident to $b$ edges of each color $k$, except for exactly one color in $C_N$.

\begin{itemize}
\item For $k\in C_O$ and $y_p\in Y$, by (v) and (vi), $deg_{G'(k)}(y_p)=1+2\left(\frac{b^2-1}{2}\right)=b^2$. By Theorem \ref{amalgamationtheorem} (i), $deg_{G(k)}(v_i)=deg_{G'(k)}(y_p)/\eta(y_p)=b^2/b=b$ for $v_i\in \phi^{-1}(y_p)$.
\item For $k\in C_N\setminus \{p\}$ and $y_p\in Y$, by (ii), (iii), and (iv), $deg_{G'(k)}(y_p)=a+1+(b^2-(a+1))=b^2$. By Theorem \ref{amalgamationtheorem} (i), $deg_{G(k)}(v_i)=deg_{G'(k)}(y_p)/\eta(y_p)=b^2/b=b$ for $v_i\in \phi^{-1}(y_p)$.
\end{itemize}

We now check that $G=K_{n_2}$ by determining there is exactly one edge between each pair of vertices in $V(G)$.

\begin{itemize}
\item We first show there is exactly $1$ edge between each pair of vertices $v_i,v_j\in V_O$. For $x_i,x_j\in X$, by (i), $m_{G'}(x_i,x_j)=1$. By Theorem \ref{amalgamationtheorem} (iii), $m_G(v_i,v_j)=m_{G'}(x_i,x_j)/(\eta(x_i)\eta(x_j))=1/((1)(1))=1$ for $v_i\in \phi^{-1}(x_i)$ and $v_j\in \phi^{-1}(x_j)$.
\item We now show there is exactly $1$ edge between each pair of vertices $v_i\in V_O$ and $v_j\in V_N$. For $x_i\in X$ and $y_p\in Y$, by (ii), $m_{G'}(x_i,y_p)=b$. By Theorem \ref{amalgamationtheorem} (iii), $m_G(v_i,v_j)=m_{G'}(x_i,y_p)/(\eta(x_i)\eta(y_p))=b/((1)(b))=1$ for $v_i\in \phi^{-1}(x_i)$ and $v_j\in \phi^{-1}(y_p)$.
\item We now show there is exactly $1$ edge between each pair of vertices $v_i,v_j\in V_N$, where $v_i\in \phi^{-1}(y_p)$ and $v_j\in \phi^{-1}(y_s)$, and $y_p,y_s\in Y$. By (iii), (iv), (v), and (vi),
\[\begin{aligned}
m_{G'}(y_p,y_s) &= \sum_{i=3}^{6}m_{G_i}(y_p,y_s)\\
&= \sum_{i=3}^{5}m_{G_i}(y_p,y_s) + m_{G_6}(y_p,y_s)\\
&= \sum_{i=3}^{5}m_{G_i}(y_p,y_s) +b^2-\sum_{i=3}^{5}m_{G_i}(y_p,y_s))\\
&= b^2.
\end{aligned}\]
By Theorem \ref{amalgamationtheorem} (iii), $m_G(v_i,v_j)=m_{G'}(y_p,y_s)/(\eta(y_p)\eta(y_s))=b^2/((b)(b))=1$ for $v_i\in \phi^{-1}(y_p)$ and $v_j\in \phi^{-1}(y_s)$.
\item Finally, we show there is exactly $1$ edge between each pair of vertices $v_i,v_j\in V_N$, where $v_i,v_j\in \phi^{-1}(y_p)$ for $y_p \in Y$. By (vi), $\ell_{G'}(y_p)=\frac{b(b-1)}{2}={b \choose 2}={\eta(y_p) \choose 2}$. By Theorem \ref{amalgamationtheorem} (ii), $m_G(v_i,v_j)=\ell_{G'}(y_p)/{\eta(y_p)\choose 2}={\eta(y_p)\choose 2}/{\eta(y_p)\choose 2}=1$ for $v_i,v_j\in \phi^{-1}(y_p)$.
\end{itemize}

\noindent
{\bf Case 2:} Assume $a$ is odd.

\begin{enumerate}[i)]
\item Define $G_1$ as done in Case 1 (i).

\item Define $G_2$ as done in Case 1 (ii).

\item Form $B_p$, $0\leq p\leq m-1$, by taking the disjoint union of $\frac{r}{2}$ copies of $K_2$ and $m-r$ copies of $K_1$ on the vertex set $Y$ using the following steps. Attach vertices $y_{p+1+j\Mod{m}}$ and $y_{p+r-j\Mod{m}}$ with an edge colored $s_1+1+p$ for $0\leq j\leq \frac{r}{2}-1$ (Since $a$ and $m-1$ are odd, by (\ref{divalg}), $r$ is even, and $r/2$ is an integer). Form $D_p$, $0\leq p\leq m-1$, by taking the disjoint union of $\frac{m-r-2}{2}$ copies of $2K_2$, $1$ copy of $C_1$, and $r+1$ copies of $K_1$ on the vertex set $Y$ using the following steps. Attach vertices $y_{p+r+j\Mod{m}}$ and $y_{p-j\Mod{m}}$  with $2$ edges colored $s_1+1+p$ for $1\leq j\leq \frac{m-r-2}{2}$ (Since $r$ is even and $m$ is even, $m-r-1$ is odd, and $\frac{m-r-2}{2}$ is an integer).
Place a loop on vertex $y_{p+\frac{m+r}{2}\Mod{m}}$ and color it with color $s_1+1+p$. 
See Figure \ref{iiiex} (b).
Let $G_3=\bigcup_{0\leq p\leq m-1}(B_p\cup D_p)$. Note that $G_3$ is regular. We claim it is ($2m-r-2$)-regular. By definition, $|E(B_p\cup D_p)|=\frac{r}{2}+2(\frac{m-r-2}{2})+1=m-\frac{r}{2}-1$. Therefore, $|E(G_3)|=m(m-\frac{r}{2}-1)$. So, $\sum_{y\in V(G_3)}deg_{G_3}(y)=2m(m-\frac{r}{2}-1)=m(2m-r-2)$. Thus, $G_3$ is ($2m-r-2$)-regular.

\item Let $G_{4}=(b^2 - (a+2))K_m$ on the vertex set $Y$. Using Lemma \ref{Kdecomp}, decompose $G_{4}$ into $m\left(\frac{b^2-(a+2)}{2}\right)$ cycles of length $m-1$. We color the edges of $G_4$ as follows. For $0\leq p\leq m-1$, color the $\frac{b^2-(a+2)}{2}$ cycles with no edges incident to $y_p$ with color $s_1+1+p$
(This is possible because each cycle excludes exactly $1$ vertex $y_p \in Y$).

\item Define $G_5$ as done in Case 1 (v).

\item Let $H=b^2K_m$ on the vertex set $Y$ with $\frac{b(b-1)}{2}$ loops on each $y\in Y$. Let $G_6=H-(G_3\cup G_4\cup G_5)$.
We first check that $G_6$ exists (we have not removed nonexistent edges from $H$).
For each pair of vertices $y_p,y_s \in Y$ there are at most $4$ edges between $y_p$ and $y_s$ in $G_3$, exactly $b^2-(a+2)$ edges between $y_p$ and $y_s$ in $G_4$, and at most $1$ edge between $y_p$ and $y_s$ in $G_5$. The minimum number of edges between $y_p$ and $y_s$ in $G_6$ is $b^2-4-(b^2-(a+2))-1=a-3$. By (\ref{a>2}), $a > 2$, and $a$ is odd, so $a-3\geq 0$ and $G_6$ exists. 

The degree of each vertex in $H$ is $b^2(m-1)+b(b-1)$, in $G_3$ is $2m-r-2$, in $G_4$ is $(b^2-(a+2))(m-1)$, and in $G_5$ is $2t_1+1$. Then, the degree of each vertex $y\in Y$ in $G_6$ is 

$\begin{array}{llll}
\deg_{G_6}(y)&=&\deg_H(y)-\deg_{G_3}-\deg_{G_4}-\deg_{G_5}\\[4pt]
&=& b^2(m-1)+b(b-1)-(2m-r-2) \\[4pt]
& & -(b^2-(a+2))(m-1)-(2t_1+1)\\[4pt]
&=& a(m-1)+r-(2t_1+1)+b(b-1)\\[4pt]
&=& n_1b-(2t_1+1)+b(b-1) & \text{By (\ref{divalg})}\\[4pt]
&=& (2t_1b+1)b-2t_1-1+b(b-1)\\[4pt]
&=& (2t_1+1)(b^2-1).\\
\end{array}$

By Theorem \ref{petersen}, $G_6$ can be partitioned into $\frac{(2t_1+1)(b^2-1)}{2}$ edge-disjoint $2$-factors. For $k\in C_O$, color $\frac{(b^2-1)}{2}$ distinct $2$-factors with color $k$.
\end{enumerate}

We define $G'=\bigcup_{i=1}^6G_i$. We will use Theorem \ref{amalgamationtheorem} to form the detachment $G=K_{n_2}$ with our desired equitable $(2t_2+1,2t_2)$-edge-coloring and an arbitrary equitable $(2t_1+1,2t_1)$-edge-coloring of $K_{n_1}$ as a subgraph. Recall that $\phi: V(G) \rightarrow V(G')$ is the amalgamation function with amalgamation numbers $\eta(x_i)= 1$ for $1 \leq i \leq n_1$ and $\eta(y_p)= b$ for $ 0 \leq p \leq  m-1$. 

We first check that there are $b$ edges of each color incident to each vertex $v_i\in V_O \cup V_N$ except exactly one color in $C_O \cup C_N$.
For $k\in C_N\setminus \{p\}$ and $y_p \in Y$, by (ii), (iii), and (iv), $deg_{G'(k)}(y_p)=a+2+(b^2-(a+2))=b^2$. By Theorem \ref{amalgamationtheorem} (i), $deg_{G(k)}(v_i)=deg_{G'(k)}(y_p)/\eta(y_p)=b^2/b=b$ for $v_i\in \phi^{-1}(y_p)$.
Checking this for all other $v_i \in V_O \cup V_N$ is the same as Case 1
(see bullets 1-3).





We now check that $G=K_{n_2}$ by determining there is exactly one edge between each pair of vertices in $v_i,v_j\in V(G)$.
For $v_i,v_j\in V_N$, where $v_i,v_j\in \phi^{-1}(y_p)$ for $y_p \in Y$, by (iii) and (vi), $\ell_{G'}(y_p)= 1+\left(\frac{b(b-1)}{2}-1\right)=\frac{b(b-1)}{2}={b \choose 2}={\eta(y_p)\choose 2}$. By Theorem \ref{amalgamationtheorem} (ii), $m_G(v_i,v_j)=\ell_{G'}(y_p)/{\eta(y_p)\choose 2}={\eta(y_p)\choose 2}/{\eta(y_p)\choose 2}=1$. Checking this for all other $v_i,v_j\in V(G)$ is the same as Case 1 (see bullets 5-7).
\end{proof}

\subsection{Intuition for Separate Cases}
\label{Intuition}

We now explain why Lemma \ref{obs:Old-color-method-theorem} or Lemma \ref{obs:New-color-method-theorem} cannot be adapted to all embeddings.

A method similar to the method in Lemma \ref{obs:Old-color-method-theorem} of amalgamating all the vertices in $V_n$ into a single vertex can be used to embed an equitable $(s_1, p_1)$-edge-coloring of $K_{n_1}$ into an an equitable $(s_2, p_2)$-edge-coloring of $K_{n_2}$ for some  $2t_1<t_2<4t_1$.
For example, we have found a method for embedding of an equitable $(3,2)$-edge-coloring of $K_7$ in an equitable $(7,6)$-edge-coloring of $K_{19}$ using a single amalgamated vertex.
However, it cannot be used for all  $2t_1<t_2<4t_1$ as seen in the following observation.
Note that the method used in Lemma \ref{obs:Old-color-method-theorem} is based on the fact that all vertices in $V_N$ are missing the same color.

\begin{observation}\label{obs:color-missing-observation}
If the same color is missing from all vertices in $V_N$, then $bt_1 \geq t_2$.
\end{observation}

\begin{proof}
 Assume that all of the vertices in $V_N$ are missing the same one color in $C_O \cup C_N$.
 By Lemma $\ref{colorvectorlemma}$ we know that every color must be missing from at least one vertex.
 Thus the rest of the $s_2 - 1$ colors must be missing from a vertex in $V_O$, implying $n_1 \geq s_2-1$.
 So, $2t_1b+1 \geq 2t_2+1-1$.
 Therefore, $bt_1 \geq t_2$ (because $b$, $t_1$, and $t_2$ are integers).
\end{proof}

Because $b \geq 3$, the method from Lemma \ref{obs:Old-color-method-theorem} can be used for embedding when $t_2=2t_1$.
However, if $t_2 > bt_1$, this method cannot be used.
Therefore, for some $t_2$ where $2t_1 < t_2 < 4t_1$, a different method is necessary.

Also, a method similar to the method in Lemma \ref{obs:New-color-method-theorem} of having all the vertices in $V_O$ missing a color in $C_O$ cannot be used if $t_2=2t_1$ as seen in the following observation.

\begin{observation}
If $t_2=2t_1$, then the vertices in $V_O$ cannot all be missing a color in $C_O$.
\end{observation}

\begin{proof}
Assume for a contradiction that each vertex in $V_O$ is missing a color in $C_O$.
Then each vertex in $V_O$ needs to be incident to each color in $C_N$.
Let $k \in C_N$.
So, there must be $n_1b=(2t_1b+1)b$ edges colored $k$ between $V_O$ and $V_N$ (since each vertex in $V_O$ needs $b$ edges colored $k$ incident to it).
Since we have $n_2-n_1= 2t_1b$ vertices in $V_N$ and each one can have at most $b$ edges colored $k$ incident to it, then there are at most $(2t_1b)b$ edges colored $k$ incident to the vertices in $V_N$.
This is a contradiction.
 \end{proof}

Therefore, Lemma \ref{obs:Old-color-method-theorem} or Lemma \ref{obs:New-color-method-theorem} cannot be adapted to cover all values $2t_1 \leq t_2 \leq 4t_1$.




\section{Main Result}
\label{mainresultsection}

We begin with a helpful observation that is fundamental in the proof of Theorem \ref{mainthm}.

\begin{observation} \label{multipleembedobs}
Suppose an equitable $(s_1,p_1)$-edge-coloring of $K_{n_1}$, say $W_1$, can be embedded in an equitable $(s_2, p_2)$-edge-coloring of $K_{n_2}$, say $W_2$.
Also, suppose that $K_{n_2}$ with coloring $W_2$ can be embedded in an equitable $(s_3, p_3)$-edge-coloring of $K_{n_3}$, say $W_3$.
Then $K_{n_1}$ with coloring $W_1$ can be embedded in $K_{n_3}$ with $W_3$.
\end{observation}

We now prove our main result.

\begin{theorem}
\label{mainthm}
An equitable $(2t_1+1, 2t_1)$-edge-coloring of $K_{n_1}$ can be embedded in an equitable $(2t_2+1, 2t_2)$-edge-coloring of $K_{n_2}$ if and only if $t_2 \geq 2t_1$.
\end{theorem}

\begin{proof}
By Corollary \ref{equitableembedinequality}, if we can embed an equitable $(2t_1+1, 2t_1)$-edge-coloring of $K_{n_1}$ into an equitable $(2t_2+1, 2t_2)$-edge-coloring of $K_{n_2}$, then $t_2 \geq 2t_1$.

By Lemma \ref{thm:b(n)=1}, if $b=1$, an equitable $(2t_1+1, 2t_1)$-edge-coloring of $K_{n_1}$ can be embedded in an equitable $(2t_2+1, 2t_2)$-edge-coloring of $K_{n_2}$.

Suppose $t_2 \geq 2t_1$ and $b \geq 3$.
Use the following steps.

\begin{enumerate}
    \item If $t_2=2t_1$, use Lemma \ref{obs:Old-color-method-theorem}.
    \item If $2t_1< t_2 < 4t_1$, use Lemma \ref{obs:New-color-method-theorem}. 
    \item If $t_2 \geq 4t_1$, then repeatedly apply Lemma \ref{obs:Old-color-method-theorem} until we are in one of the two cases above.
\end{enumerate}
By Observation \ref{multipleembedobs} we obtain our desired embedding.
\end{proof}

\section{Future Directions}
\label{futuresection}

In this paper, we answered Question \ref{mainquestbpreserved}, which was a subquestion of Question \ref{mainquest}.
In this section we describe other natural subquestions of Question \ref{mainquest}.

Consider the scheduling problem from the introduction. Suppose instead of preserving the rate of meetings per hour, the length of the conference stays constant. This corresponds to cases where $b$ increases and $s$ stays the same. 

\begin{question}
\label{questionfixingp}
Suppose $n_1$ and $n_2$ are integers and $n_1<n_2$ and $p \leq s$.
Can we embed an equitable $(s, p)$-edge-coloring of $K_{n_1}$ in an equitable $(s, p)$-edge-coloring of $K_{n_2}$?
\end{question}

Another possible direction is to modify Question \ref{mainquest} for $n_1$ and/or $n_2$ to be in the other cases described in Theorem \ref{mainlambda1}.
Therefore, in the following question, we no longer assume $n_i=4t_ix+2t_i+1$.

\begin{question}
\label{questionothern}
Suppose $n_1<n_2$, $1 \leq p_1 \leq n_1-1$, $1 \leq p_2 \leq n_2-1$, 
and $p_1 \leq p_2$ for integers $n_1$, $n_2$, $p_1$, and $p_2$.
When can we embed an equitable $(\chi_{p_1}(n_1), p_1)$-edge-coloring of $K_{n_1}$ in an equitable $(\chi_{p_2}(n_2), p_2)$-edge-coloring of $K_{n_2}$, say $W_2$ with colors $C$, so that $b(n_2, W_2, u,i)= \lceil \frac{n_1-1}{p_1} \rceil$ or $b(n_2, W_2, u,i)= \lceil \frac{n_1-1}{p_1} \rceil-1$
for all $u \in V(K_{n_2})$ and $i \in C_2$?
\end{question}

This problem can be split into the $4$ cases below.
Note that we considered changing ``$b(n_2, W_2, u,i)= \lceil \frac{n_1-1}{p_1} \rceil-1$'' to ``$b(n_2, W_2, u,i)= \lfloor \frac{n_1-1}{p_1} \rfloor$'' in the previous question.
However, that would eliminate Case 3 below.\\

\noindent {\bf Case 1:} Assume $n_1 \equiv p_1+1 \Mod{2p_1}$ and $n_2  \equiv p_2+1 \Mod{2p_2}$.
This case is solved in Theorem \ref{mainthm}.\\

\noindent {\bf Case 2:} Assume $n_1 \not\equiv p_1+1 \Mod{2p_1}$ and $n_2 \not\equiv p_2+1 \Mod{2p_2}$.  
This case is interesting because, unlike Case 1, the original coloring of $K_{n_1}$ impacts whether the graph can be embedded in $K_{n_2}$.
Consider the example of embedding an equitable $(2,2)$-edge-coloring of $K_4$ into an equitable $(2,2)$-edge-coloring of $K_5$. For one such coloring of $K_4$, $H_1$ (see Figure \ref{embedK4}), this embedding is possible, but for another equitable $(2,2)$-edge-coloring of $K_4$, $H_2$ (see Figure \ref{embedK4}), the embedding is not possible. \\

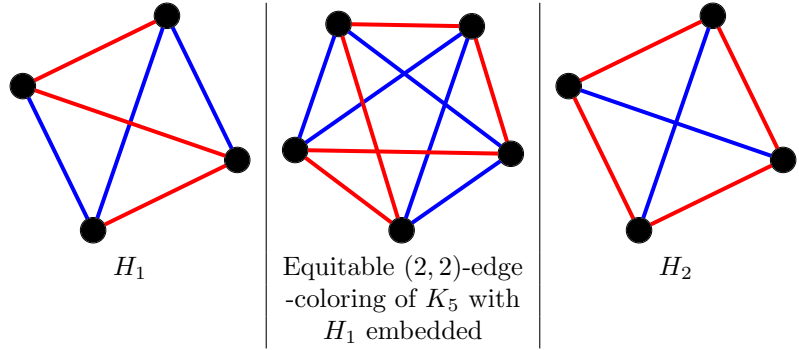
\begin{figure}
\centering
\begin{tabular}{c|c|c}
    \begin{tikzpicture}[scale=0.5, every node/.style={fill=black, circle, font=\scriptsize, very thin, scale=.6}]

\drawlabels[4]

\draw[line width=1.5pt, red] (1) -- (2);
\draw[line width=1.5pt, blue] (1) -- (3);
\draw[line width=1.5pt, blue] (1) -- (4);

\draw[line width=1.5pt, blue] (2) -- (3);
\draw[line width=1.5pt, red] (2) -- (4);

\draw[line width=1.5pt, red] (3) -- (4);

\end{tikzpicture} & \begin{tikzpicture}[scale=0.5, every node/.style={fill=black, circle, font=\scriptsize, very thin, scale=.6}]

\drawlabels[5]

\draw[line width=1.5pt, red] (1) -- (2);
\draw[line width=1.5pt, blue] (1) -- (3);
\draw[line width=1.5pt, blue] (1) -- (4);
\draw[line width=1.5pt, red] (1) -- (5);

\draw[line width=1.5pt, blue] (2) -- (3);
\draw[line width=1.5pt, red] (2) -- (4);
\draw[line width=1.5pt, blue] (2) -- (5);

\draw[line width=1.5pt, red] (3) -- (4);
\draw[line width=1.5pt, red] (3) -- (5);

\draw[line width=1.5pt, blue] (4) -- (5);

\end{tikzpicture} &  \begin{tikzpicture}[scale=0.5, every node/.style={fill=black, circle, font=\scriptsize, very thin, scale=.6}]

\drawlabels[4]

\draw[line width=1.5pt, red] (1) -- (2);
\draw[line width=1.5pt, blue] (1) -- (3);
\draw[line width=1.5pt, red] (1) -- (4);

\draw[line width=1.5pt, red] (2) -- (3);
\draw[line width=1.5pt, blue] (2) -- (4);

\draw[line width=1.5pt, red] (3) -- (4);

\end{tikzpicture}\\
$H_1$&Equitable $(2,2)$-edge&$H_2$\\
&-coloring of $K_5$ with&\\
&$H_1$ embedded&\\
\end{tabular}
\caption{Coloring dependent embedding of $K_4$ into $K_5$}
\label{embedK4}
\end{figure}

\noindent {\bf Case 3:}  Assume $n_1 \equiv p_1+1 \Mod{2p_1}$ and $n_2 \not\equiv p_2+1 \Mod{2p_2}$. 
Similarly to Case 2, the original coloring of $K_{n_1}$ impacts whether the graph can be embedded in $K_{n_2}$.
For example, an equitable $(5,4)$-edge-coloring of $K_{13}$ with color vector $(10,10,10,10,12)$ can be embedded in an equitable $(5,5)$-edge-coloring of $K_{15}$, but a coloring of $K_{13}$ with color vector $(4,12,12,12,12)$ cannot be embedded.
Note that these equitable $(5,4)$-edge-colorings of $K_{n_1}$ exist by \cite{MR3774803} and \cite{MR3764373}.
We exclude a detailed explanation for the sake of brevity.\\

\noindent {\bf Case 4:} Assume $n_1 \not\equiv p_1+1 \Mod{2p_1}$ and $n_2 \equiv p_2+1 \Mod{2p_2}$.
There does not seem to be the same coloring dependence as in  Cases 2 and 3.
Therefore, we believe this case can be done in a similar manner to Case 1.
\\

\section*{Acknowledgments}

The work was supported by the College of Charleston Office of Undergraduate Research and Creative Activities grants no. SU2024-21 and 031 and the College of Charleston School of Engineering, Computing and Mathematics.
The authors thank Chris Rodger for helpful discussion about the project.


\section* {Declaration of Generative AI and AI-Assisted Technologies in the Manuscript Preparation Process}
During the preparation of this work the authors used ChatGPT and Grammarly in order to help create LaTex images and correct grammar respectively. After using this tool/service, the author(s) reviewed and edited the content as needed and take full responsibility for the content of the published article.


\bibliographystyle{plain}
\bibliography{equitablebib.bib}

\end{document}